\documentclass[reqno]{amsart}

\usepackage{enumerate}
\usepackage{amssymb}
\usepackage[colorlinks=true]{hyperref}
\hypersetup{citecolor=blue}
\usepackage{color}

\numberwithin{equation}{section}

\newtheorem{thrm}{Theorem}[section]
\newtheorem{lmm}[thrm]{Lemma}
\newtheorem{crllr}[thrm]{Corollary}
\newtheorem{prpstn}[thrm]{Proposition}
\theoremstyle{definition}
\newtheorem{rmrk}[thrm]{Remark}

\newcommand\R{{\mathbb R}}

\newcommand\Comp{{\mathrm{c}}}

\newcommand\Ens{{\mathcal E}}

\newcommand\dist{{\mathrm{d}}}

\newcommand\Supp{{\mathrm{supp}}\, }

\newcommand\Md{{\mathcal M}}

\newcommand\Kernel{{\mathcal K}}

\newcommand\CSTu{\boldsymbol {A}}

\newcommand\CSTtu{\boldsymbol {B}_1}
\newcommand\CSTtd{\boldsymbol {B}_2}

\newcommand\KIR{1 _{ \{  |y|<r\} } }

\newcommand\CSTK{\boldsymbol {M}}

\newcommand\Loc{{\mathrm{loc}}}

\newcommand\goto{\mathop{\longrightarrow}}

\newcommand\Step[1]{\medskip \noindent {\bf Step~#1.}\quad}

\newcommand\MScN[1]{\href{http://www.ams.org/mathscinet-getitem?mr=#1}{\nolinkurl{(#1)}}}
\newcommand\DOI[1]{\href{http://dx.doi.org/#1}{(doi: \nolinkurl{#1})}}
\newcommand\LINK[1]{\href{#1}{(link: \nolinkurl{#1})}}

\newcommand\DI{u_0 }

\newcommand\DIbd{w_0 }

\newcommand\DFNk{\omega }

\newcommand\CSTB[1]{B_{#1} }
\newcommand\CSTR[1]{R_{#1} }
\newcommand\cstpa{a}

\definecolor{bostonuniversityred}{rgb}{0.8, 0.0, 0.0}

\newcommand\xHn[1]{ {\mathrm{H}}^{{#1}}}
\newcommand\xLn[1]{ {\mathrm{L}}^{{#1}}}
\newcommand\xCn[1]{ {\mathrm{C}}^{{#1}}}
\newcommand\xWn[1]{ {\mathrm{W}}^{{#1}}}
\newcommand\xdif{\,\mathrm{d}}

\begin{document}

\title{Sign-changing solutions of the nonlinear heat equation with persistent singularities}

\def\shorttitle{Sign-changing solutions with persistent singularities}

\author[T. Cazenave]{Thierry Cazenave$^1$}
\email{\href{mailto:thierry.cazenave@sorbonne-universite.fr}{thierry.cazenave@sorbonne-universite.fr}}

\author[F. Dickstein]{Fl\'avio Dickstein$^{1,2}$}
\email{\href{mailto:flavio@labma.ufrj.br}{flavio@labma.ufrj.br}}

\author[I. Naumkin]{Ivan Naumkin$^3$}
\email{\href{mailto:ivan.naumkin@iimas.unam.mx}{ivan.naumkin@iimas.unam.mx}}

\author[F. B.~Weissler]{Fred B. Weissler$^4$}
\email{\href{mailto:weissler@math.univ-paris13.fr}{weissler@math.univ-paris13.fr}}

\address{$^1$Sorbonne Universit\'e, CNRS, Universit\'e de Paris. Laboratoire Jacques-Louis Lions,
B.C. 187, 4 place Jussieu, 75252 Paris Cedex 05, France}

\address{$^2$Instituto de Matem\'atica, Universidade Federal do Rio de Janeiro, Caixa Postal 68530, 21944--970 Rio de Janeiro, R.J., Brazil}

\address{$^3$Departamento de F\'{\i}sica Matem\'{a}tica, Instituto de Investigaciones
en Matem\'{a}ticas Aplicadas y en Sistemas. Universidad Nacional Aut\'{o}noma
de M\'{e}xico, Apartado Postal 20-126, Ciudad de M\'{e}xico, 01000, M\'{e}xico.}

\address{$^4$Universit\'e Sorbonne Paris Nord, CNRS UMR 7539 LAGA, 99 Avenue J.-B. Cl\'e\-ment, F-93430 Villetaneuse, France}

\subjclass[2010] {Primary 35K91; secondary 35K58, 35C06, 35K67, 35A01, 35A21}

\keywords{Nonlinear heat equation, sign-changing solutions, singular self-similar solutions, singular stationary solutions, persistent singularities}

\thanks{Research supported by the ``Brazilian-French Network in Mathematics"}
\thanks{Fl\'avio Dickstein was partially supported by CNPq (Brasil).}
\thanks{Ivan Naumkin is a Fellow of Sistema Nacional de Investigadores. He was partially supported by project PAPIIT IA101820}

\begin{abstract}
We study the existence of sign-changing solutions to the nonlinear heat equation $\partial _t u = \Delta u +  |u|^\alpha u$  on ${\mathbb R}^N $, $N\ge 3$, with $\frac {2} {N-2} < \alpha <\alpha _0$, where $\alpha _0=\frac {4} {N-4+2\sqrt{ N-1 } }\in (\frac {2} {N-2}, \frac {4} {N-2})$, which are singular at $x=0$ on an interval of time. 
In particular, for certain $\mu >0$ that can be arbitrarily large, we prove that for any $u_0 \in \mathrm{L} ^\infty _{\mathrm{loc}} ({\mathbb R}^N \setminus \{ 0 \}) $ which is bounded at infinity and equals  $\mu  |x|^{- \frac {2} {\alpha }}$ in a neighborhood of $0$, there exists a local (in time) solution $u$ of the nonlinear heat equation with initial value $u_0$, which is sign-changing, bounded at infinity and has the singularity $\beta  |x|^{- \frac {2} {\alpha }}$ at the origin in the sense that for $t>0$, $ |x|^{\frac {2} {\alpha }} u(t,x) \to \beta  $ as $ |x| \to 0$, where $\beta =  \frac {2} {\alpha }   ( N -2 - \frac {2} {\alpha } ) $. 
These solutions in general are neither stationary nor self-similar. 
\end{abstract}

\maketitle

\section{Introduction}
In this paper, we study the nonlinear heat equation 
\begin{equation} \label{NLHE}
\partial _t u = \Delta u +  |u|^\alpha u
\end{equation}
on $\R^N $ with 
\begin{equation} \label{fCondAlpha} 
N\ge 3 \quad  \text{and}\quad  \frac {2} {N-2}<\alpha <\alpha _0, 
\end{equation} 
where 
\begin{equation} \label{fPS0} 
\alpha _0=\frac {4} {N-4+2\sqrt{ N-1 } }.
\end{equation}
(Note that $\frac {2} {N-2} < \alpha _0 < \frac {4} {N-2}$, see Lemma~\ref{eNUM1}~\eqref{eNUM1:1b1}.) 
We are interested in sign-changing solutions of~\eqref{NLHE} which have a singularity at $x=0$ on an interval of time. 

Positive solutions of~\eqref{NLHE} with a standing or moving singularity have been well studied. 
The simplest, for all $\alpha >\frac {2} {N-2}$, is the homogeneous stationary solution $  \beta ^{\frac {1} {\alpha }} |x|^{- \frac {2} {\alpha }}$ where 
\begin{equation} \label{bta}
\beta = \frac {2} {\alpha }  \Bigl( N -2 - \frac {2} {\alpha } \Bigr) >0 .
\end{equation}
Moreover, for all $\frac {2} {N-2} < \alpha < \frac {4} {N-2} $,
\eqref{NLHE} has a one-parameter family $(U _\lambda )  _{ \lambda >0 } \subset \xCn{2} (\R^N \setminus \{ 0\})$ of radially symmetric, positive, singular stationary solutions satisfying $ |x|^{\frac {2} {\alpha }} U_\lambda (x) \to \beta ^{\frac {1} {\alpha }}  $ as $x\to 0$ and $ |x|^{N-2} U_\lambda (x) \to \lambda $ as $ |x| \to \infty $. 
Furthermore, for this range of $\alpha $, the family $(U _\lambda )  _{ \lambda >0 }$ and $  \beta ^{\frac {1} {\alpha }} |x|^{- \frac {2} {\alpha }}$ constitute all the positive, radially symmetric, singular stationary solutions of~\eqref{NLHE}. See~\cite[Proposition~3.1]{SerrinZ}.
Under the stronger assumption~\eqref{fCondAlpha}, these solutions can be used as prototypes to construct positive solutions of~\eqref{NLHE} with a moving singularity, i.e. a singularity located at $x= \xi (t)$ for every $t$ in some interval, under appropriate conditions on the function $\xi (\cdot )$.
See~\cite{SatoY2009, SatoY2012-2}. 
Positive self-similar solutions of~\eqref{NLHE}, both forward and backward,  with a standing or moving singularity, 
have also been constructed, see~\cite{SatoY2010, SatoY2012-1}. 
The finite-time blowup and the long-time asymptotic behavior of positive singular solutions of~\eqref{NLHE} are studied in~\cite{Sato2011, HoshinoY}. 

Sign changing stationary solutions of equation~\eqref{NLHE} have been less studied. We show here that  for all $\frac {2} {N-2} < \alpha < \frac {4} {N-2} $, equation~\eqref{NLHE}  has sign-changing, radially symmetric, stationary solutions that are singular at $x=0$. These solutions behave like $  \beta ^{\frac {1} {\alpha }} |x|^{- \frac {2} {\alpha }}$ at the origin and oscillate indefinitely as $ |x|\to \infty $. See Proposition~\ref{eStat1} and Corollary~\ref{eRemSum}.

For the same range $\frac {2} {N-2} < \alpha < \frac {4} {N-2} $, equation~\eqref{NLHE} also has sign-changing, radially symmetric, self-similar solutions which  are singular for all positive time.
More precisely, it follows from~\cite[Theorem~1.3]{CDNW1} that there exist an integer $ \overline{m}\ge 0 $ and an increasing sequence $( \mu _m) _{ m\ge   \overline{m}  }\subset (0, \infty )$, $\mu _m \to \infty $ as $m\to \infty $, such that for each $\mu _m$ there exists a radially symmetric self-similar solution 
\begin{equation*} 
U  (t, x)= t^{- \frac {1} {\alpha }} f  \Bigl( \frac { |x|} {\sqrt t} \Bigr)
\end{equation*}
of~\eqref{NLHE}  in the sense of distributions, where the profile $f \in \xCn{2} (0, \infty ) $ has exactly $m$ zeros and satisfies 
\begin{equation} \label{fpr2}
f'' + \Big(\frac{N-1}{r} +  \frac{r}{2}\Big)f'  + \frac{1}{\alpha}f  +  | f |^{\alpha}f = 0 .
\end{equation} 
Moreover, $r^{\frac {2} {\alpha }} f (r) \to  \beta ^{\frac {1} {\alpha }}$ as $r \to 0$ and $r^{\frac {2} {\alpha }} f (r) \to (-1)^m \mu _m$ as $r\to \infty $. 
It follows that  $|x|^{\frac {2} {\alpha }}U (t,x) \to  \beta ^{\frac {1} {\alpha }} $ as $ |x | \to 0$ for all $t>0$, and 
$ U (t, \cdot ) \to  (-1)^m \mu _m  | \cdot |^{- \frac {2} {\alpha }} $ in $\xLn{1}_\Loc  (\R^N ) $ as $t\to 0$.

The purpose of this article is to construct sign-changing solutions of~\eqref{NLHE} which are singular at $x=0$ for small positive time, and which are neither stationary nor self-similar. 
The construction of these solutions is based on the following perturbation result.

\begin{thrm} \label{ePSSZ1} 
Assume~\eqref{fCondAlpha}. 
Let $S>0$ and 
\begin{equation} \label{ePSSZ2} 
U \in  \xLn{{\alpha + 1}}_\Loc ((0, S ) \times \R^N ) \cap \mathrm{C} (( 0,S) \times ( \R^N \setminus \{0\}  ) ) 
\end{equation} 
satisfy for some constant $C$
\begin{equation} \label{fPShrho} 
| \,  |x|^\frac {2} {\alpha }U(t, x )-\beta ^{\frac {1} {\alpha }}|\le C\Bigl [\Bigl (\frac {  |x|} {  |x|+\sqrt{ t } } \Bigr )^\rho + |x|^\frac {2} {\alpha } \Bigr] ,
\end{equation}  
 for $0<t<S$ and $x\not = 0$, where
\begin{equation} \label{fPSrho} 
\rho = \frac {2} {\alpha } - \frac {N-2} {2} - \sqrt{   \frac {(N-2)^2} {4} - \beta (\alpha +1)   } ,
\end{equation}
and $\beta $ is given by~\eqref{bta}. 
 ($\rho >0$ by Lemma~$\ref{eNUM1}~\eqref{eNUM1:3}$  below.)
Assume further that there exists $U_0 \in \xLn{1} _{ \Loc } (\R^N ) $ such that
\begin{equation} \label{fDfnUdz} 
U(t, \cdot ) \goto _{ t\to 0 } U_0 (\cdot ),
\end{equation} 
in $\xLn{1} _{ \Loc } (\R^N ) $. 
 Given $  \delta >0$ and $\DI \in \xLn{1} _\Loc   (\R^N  ) \cap \xLn{\infty }   ( \{ |x| >\delta  \} ) $ such that
\begin{equation*} 
\DI (x) = U_0 (x)  \text{ a.e. on } \{  |x|<\delta  \},
\end{equation*} 
there exist $ T \in (0,S)$ and  a solution  $u\in \xLn{{\alpha +1}} _\Loc ((0,T) \times \R^N ) \cap \mathrm{C} (( 0,T) \times ( \R^N \setminus \{0\}  ) )$ of
\begin{equation} \label{fPS4:b8} 
\partial _t u - \Delta u =  |u|^\alpha u - [ \partial _t U - \Delta U -  |U|^\alpha U] ,
\end{equation} 
in the sense of distributions ${\mathcal D}' ((0,T) \times \R^N )$, such that
\begin{equation} \label{fPS4:b6} 
u(t) \goto _{ t\to 0 } \DI \quad  \text{in } \xLn{1} _\Loc ( \R^N ) ,
\end{equation} 
and 
\begin{equation} \label{fPS4:b7} 
 | u(t,x) - U (t, x)  |\le C  (1 +  |x|^{- \eta}), \quad 0<t<T, x\not = 0, 
\end{equation} 
where 
\begin{equation} \label{fPSeta}
\eta =  \frac {N-2} {2} -   \sqrt{  \frac {(N-2)^2} {4} -\beta (\alpha +1) } >0. 
\end{equation} 
\end{thrm} 

In Theorem~\ref{ePSSZ1}, the choice of $\DI $ is both flexible and rigid. On the one hand, for $ |x| > \delta $ there is complete freedom to choose $\DI $ as long as it is bounded. On the other hand, for $ |x|< \delta $, $\DI $ must agree precisely with $U_0$, the initial value of the given function $U$. 
The following remark gives information on the relationship between $U$ and $u$, and how $u_0$ affects this relationship. 

\begin{rmrk} \label{ePSSZ1:rem1} 

\begin{enumerate}[{\rm (i)}] 

\item \label{ePSSZ1:rem1:1} 
In Theorem~\ref{ePSSZ1},  $U$ need not be a solution of~\eqref{NLHE}, so that $u $  need not be a solution of~\eqref{NLHE}. However, if $U$ solves~\eqref{NLHE}, then so does $u $. 

\item \label{ePSSZ1:rem1:2} 
In Theorem~\ref{ePSSZ1},  $U$ need not be radially symmetric. Even if $U$ is radially symmetric, $u$ is not radially symmetric if $\DI$ is not. (Recall that $\DI$ is specified only in the ball of radius $\delta $.)

\item \label{ePSSZ1:rem1:3} 
The function $u$ cannot be a stationary solution of~\eqref{NLHE}, unless $\DI$ is a stationary solution of~\eqref{NLHE}, by~\eqref{fPS4:b6}. 

\item \label{ePSSZ1:rem1:4} 
The function $u$ cannot be a self-similar solution of~\eqref{NLHE}, unless $\DI$ is homogeneous. Indeed, recall that the initial value of a self-similar solution, if it exists, is always homogeneous. 

\item \label{ePSSZ1:rem1:5} 
Since $ \eta < \frac {N-2} {2}< \frac {2} {\alpha }$ by~\eqref{fCondAlpha} and~\eqref{eNUM1:1}, it follows from~\eqref{fPShrho} and~\eqref{fPS4:b7} that $u$ given by Theorem~\ref{ePSSZ1} is singular for all $0<t<T$ at $x=0$, and has the same singular behavior as $U(t)$, i.e.
\begin{equation} \label{fSBu} 
 |x| ^{ \frac {2} {\alpha } }  u(t, x )  \goto  _{  x \to 0 } \beta .
\end{equation} 

\end{enumerate} 
\end{rmrk} 

In order to use Theorem~\ref{ePSSZ1} to construct sign-changing solutions of~\eqref{NLHE} with a singularity at $x=0$, we consider separately the cases where $U$ is a radially symmetric stationary solution of~\eqref{NLHE}, and where $U$ is a radially symmetric self-similar solution of~\eqref{NLHE} with singular profile.
This gives the following two theorems. 

\begin{thrm} \label{ePSSZpr1} 
Assume~\eqref{fCondAlpha}. 
Let $U$ be a radially symmetric, stationary solution of~\eqref{NLHE}  that is singular at $x=0$.
Given $\delta >0$ and $\DI \in \xLn{1} _\Loc   (\R^N  ) \cap \xLn{\infty }   ( \{ |x| >\delta  \} ) $ such that
\begin{equation*} 
\DI (x) = U(x) \text{ a.e. on } \{  |x|<\delta  \},
\end{equation*} 
there exist $ T >0 $ and  a solution  $u\in \xLn{{\alpha +1}} _\Loc ((0,T) \times \R^N ) \cap \mathrm{C} (( 0,T) \times ( \R^N \setminus \{0\}  ) )$ of~\eqref{NLHE} in the sense of distributions, such that~\eqref{fPS4:b7}  holds and $u(0)= \DI$ in the sense~\eqref{fPS4:b6}. 
Moreover, $u(t)$ is singular at $x=0$ for all $t<T$ and satisfies~\eqref{fSBu}.
\end{thrm}

We stress the fact that there exist sign-changing stationary solutions $U$ to which Theorem~\ref{ePSSZpr1} applies, by Corollary~\ref{eRemSum}. In this case, if $\delta $ is sufficiently large, then the solution $u$ is sign-changing for small time by~\eqref{fPS4:b6}. See below for further discussion of this point.

\begin{thrm} \label{ePSSZpr2} 
Assume~\eqref{fCondAlpha}. 
Let $f\in \xCn{2} (0,\infty  ) $ be a solution of the equation~\eqref{fpr2}  having the singularity 
$r^{\frac {2} {\alpha }} f (r) \to  \beta ^{\frac {1} {\alpha }}$ as $r \to 0$, let
\begin{equation} \label{fEximu} 
\mu = \lim  _{ r\to \infty  } r^{\frac {2} {\alpha }} f(r)  \in \R,
\end{equation} 
which exists by Proposition~$\ref{eLemCV2}$. 
Let
\begin{equation} \label{fPS4:b3} 
U  (t, x)= t^{- \frac {1} {\alpha }} f  \Bigl( \frac { |x|} {\sqrt t} \Bigr), \quad t>0, x\not = 0,
\end{equation}
so that $U$ is a self-similar solution of~\eqref{NLHE} by Proposition~$\ref{eLemCV2}$.
Given $\delta >0$ and $\DI \in \xLn{1} _\Loc   (\R^N  ) \cap \xLn{\infty }   ( \{ |x| >\delta  \} ) $ such that
\begin{equation*} 
\DI (x) =  \mu  |x|^{- \frac {2} {\alpha }}  \text{ a.e. on } \{  |x|<\delta  \},
\end{equation*} 
there exist $ T >0 $ and  a solution  $u\in \xLn{{\alpha +1}} _\Loc ((0,T) \times \R^N ) \cap \mathrm{C} (( 0,T) \times ( \R^N \setminus \{0\}  ) )$ of~\eqref{NLHE} in the sense of distributions, such that~\eqref{fPS4:b7}  holds and $u(0)= \DI$ in the sense~\eqref{fPS4:b6}. 
Moreover, $u(t)$ is singular at $x=0$ for all $t<T$ and satisfies~\eqref{fSBu}.
\end{thrm}

There exist sign-changing  self-similar solutions $U$ to which Theorem~\ref{ePSSZpr2} applies, by Proposition~\ref{eLemCV2}.
In this case, the solution $u$ is necessarily sign-changing for small time, see the discussion below. 
Moreover, Proposition~\ref{eLemCV2} states that~\eqref{fEximu} can be achieved by a sign-changing profile for a sequence $\mu = \mu _n \to \infty $.

Note that, for a given solution $U$, Theorems~\ref{ePSSZpr1} and~\ref{ePSSZpr2} produce many different solutions of~\eqref{NLHE}  with the same singularity. Indeed, we can choose  $\DI$ arbitrarily for $ |x|>\delta $ as long as $ \DI \in \xLn{\infty }  (\{  |x|>\delta \}  )$. That two different choices of $\DI$ produce two different solutions of~\eqref{NLHE} follows from~\eqref{fPS4:b6}. 

We observe that using Remark~\ref{ePSSZ1:rem1} and Theorems~\ref{ePSSZpr1} and~\ref{ePSSZpr2} we do indeed obtain sign-changing solutions of~\eqref{NLHE} with persistent singularities, which are neither stationary nor self-similar. 
We may assume that $u_0$ is neither homogeneous, nor a stationary solution of~\eqref{NLHE}. 
That $u$ is sign-changing is of course true by~\eqref{fPS4:b6} if $\DI$ is sign-changing, and
this is always possible since $\DI$ is prescribed only for $ |x|<\delta $.
Furthermore, in Theorem~\ref{ePSSZpr1}, if $U$ is sign-changing and $\delta $ is  chosen sufficiently large, then $\DI$ is necessarily sign-changing. Finally, even if $\DI >0$, Theorem~\ref{ePSSZpr2} produces sign-changing solutions. 
Indeed, suppose the profile $f$ is sign-changing (there exist such profiles, see~\cite[Theorem~1.3]{CDNW1}), and
let $\tau, \varepsilon  >0$ be such that $f(\tau )= -\varepsilon  $. Given any $x_0\in \R^N $ such that $ |x_0|= 1$, it follows from~\eqref{fPS4:b7} with $x= \tau \sqrt t x_0$ that
\begin{equation*} 
 | t^{\frac {1} {\alpha }} u(t, \tau \sqrt t x_0) + \varepsilon   | = t^{\frac {1} {\alpha }} | u(t, \tau \sqrt t x_0) - U (t, \tau \sqrt t x_0)  | \le C ( t^{\frac {1} {\alpha }} + \tau ^{- \eta } t^{\frac {1} {\alpha} - \frac {\eta} {2}}) \le \frac {\varepsilon  } {2}
\end{equation*} 
for $t>0$ small, so that $ u(t, \tau \sqrt t x_0) \le - \frac {\varepsilon  } {2} t^{ - \frac {1} {\alpha }} <0$ for $t>0$ small.

We formalize some of the previous observations with the following corollary to Theorem~\ref{ePSSZpr2}.

\begin{crllr} \label{ePSSZpr2:cor} 
There exists a sequence $(\mu_n )  _{ n\ge 1 } \subset (0,\infty )$, $\mu _n\to \infty $, such that if 
 $\DI \in \xLn{1} _\Loc   (\R^N  ) \cap \xLn{\infty }   ( \{ |x| >1  \} ) $ satisfies
\begin{equation*} 
\DI (x) =  \mu _n |x|^{- \frac {2} {\alpha }}  \text{ a.e. on } \{  |x|<\delta  \},
\end{equation*} 
for some $\delta >0$ and $n\ge 1$, then there exist $ T >0 $ and  a sign-changing solution  $u\in \xLn{{\alpha +1}} _\Loc ((0,T) \times \R^N ) \cap \mathrm{C} (( 0,T) \times ( \R^N \setminus \{0\}  ) )$ of~\eqref{NLHE} in the sense of distributions, which is not self-similar (or stationary), such that $u(0)= \DI$ in the sense~\eqref{fPS4:b6} and $u$ satisfies~\eqref{fSBu} for all $0<t<T$. In particular, $u(t)$ is singular at $x=0$ for all $0 < t<T$. 
\end{crllr} 

In Theorem~\ref{ePSSZpr2}, the solution $u (t) $ has the spatial singularity $\mu  |x|^{- \frac {2} {\alpha }}$ when $t=0$, and the singularity $\beta ^{\frac {1} {\alpha }}  |x|^{- \frac {2} {\alpha }}$ when $t>0$. Since $\mu $ can be chosen arbitrarily large by Proposition~\ref{eLemCV2}, we see that the singularity of $u$ at $t=0$ can be greater than the singularity at  $t>0$. 

Let $\alpha >0$ and let $\DI \in \xLn{\infty }  _\Loc (\R^N \setminus \{0\}) \cap \xLn{\infty }  ( \{  |x|>1 \} )  $  equal $ \mu  |x|^{- \frac {2} {\alpha }}$ near the origin with $\mu >0$.
If $\mu $ is sufficiently large, then there is no positive (possibly singular) solution of~\eqref{NLHE} with the initial value $\DI$. See~\cite[Proposition~A.1]{CDNW1} or~\cite[Corollary~2.7]{CDNW2}.
(Note that this is not in contradiction with the results in~\cite{SatoY2009, SatoY2010, SatoY2012-1, SatoY2012-2}, since all the positive, singular solutions of~\eqref{NLHE} constructed there have an initial value $\DI = u( 0 )$ which behaves like $  \beta ^{\frac {1} {\alpha }} |x|^{- \frac {2} {\alpha }}$ near the origin.)
On the other hand, if $\alpha <\frac {4} {N-2}$, then there exist sign-changing, local in time solutions (regular for positive time) of~\eqref{NLHE} with the initial value $\DI$. See~\cite[Theorem~5.1]{CDNW2}. It follows from Corollary~\ref{ePSSZpr2:cor} that, at least for certain arbitrarily large $\mu $ and under assumption~\eqref{fCondAlpha}, there also exist local in time, sign-changing solutions of~\eqref{NLHE} with the initial value $\DI$, which are singular at the origin for positive time. 

We observe that for positive data ($U\ge 0$ and $\DI \ge 0$), Theorem~\ref{ePSSZpr1} is weaker than~\cite[Theorem~1.1]{SatoY2012-2}. Indeed, in~\cite[Theorem~1.1]{SatoY2012-2}, the singularity of $u$ can move with time, and $\DI$ need not be equal to $U(0)$ in a neighborhood of the origin, but sufficiently close to $U(0)$.  
A technical reason for this difference is that Theorem~\ref{ePSSZpr1} allows sign-changing solutions so that we cannot apply the powerful comparison arguments used in~\cite{SatoY2012-2}.

Also, it is natural ask if $u(t,x) \to 0$ as $ |x| \to \infty $ in Theorems~\ref{ePSSZpr1} and~\ref{ePSSZpr2}, assuming $\DI (x) \to 0$ as $ |x| \to \infty $. Our construction of the solution $u$ does not answer this question. 
The analogous property for perturbations of self-similar solutions with regular profile is true, see~\cite[Theorem~5.1]{CDNW2}.

We now describe our strategy to prove Theorem~\ref{ePSSZ1}.
We construct $u$ as a perturbation of $U$ in the form
\begin{equation*} 
u = U + w .
\end{equation*}
The resulting equation for $w$ is
\begin{equation} \label{fIntro1} 
\partial _t w - \Delta w =  |U + w|^\alpha (U+w) -  |U|^\alpha U .
\end{equation} 
The leading term on the right-hand side of~\eqref{fIntro1} is $(\alpha +1)  |U|^\alpha w$,  which by~\eqref{fPShrho} behaves like  $  \beta (\alpha +1)  |x|^{-2} w$ near the origin. This makes it delicate to apply a standard perturbation argument to~\eqref{fIntro1}. 
It turns out to be helpful to subtract  the term $ \beta (\alpha +1)  |x|^{-2} w $ from both sides of the equation, leading to the following heat equation with inverse square potential
\begin{equation} \label{fPT2} 
\partial _t w - \Delta w - \beta (\alpha +1)  |x|^{-2} w = \Md w, 
\end{equation} 
where
\begin{equation} \label{fPT3}
\Md w=   | U+w|^\alpha  (U+w) -   |U|^\alpha U  - \beta  (\alpha +1)  |x |^{-2} w .
\end{equation}
We observe that the operator $-\Delta -  \beta (\alpha +1)  |x|^{-2}$ in~\eqref{fPT2} has good properties only if 
\begin{equation}  \label{fPShar}
\beta  (\alpha +1) < \frac {(N-2)^2} {4}, 
\end{equation} 
 the constant in Hardy's inequality. 
Inequality~\eqref{fPShar} is equivalent to $\alpha <\alpha _0$, where $\alpha _0$ is given by~\eqref{fPS0}. 
(See Lemma~\ref{eNUM1} below.) Note that $\alpha _0> \frac {2} {N-2}$, since $N >2$. 
Under the assumption~\eqref{fPShar},  the operator $H$ on $\xLn{2} (\R^N ) $ defined by
\begin{equation} \label{fPShar:b1}
\begin{cases} 
D(H)= \{ u\in \xHn{1}  (\R^N ) ;\, \Delta u + \beta (\alpha +1)  |x|^{-2} u\in \xLn{2} (\R^N )  \}   \\ H u= \Delta u + \beta  (\alpha +1)  |x|^{-2} u ,\quad  u\in D(H)
\end{cases} 
\end{equation} 
 is a negative self-adjoint  operator, hence the generator of a $C_0$ semigroup of contractions $(e^{t H}) _{ t\ge 0 }$, which is an analytic semigroup on $\xLn{2} (\R^N ) $. 
Moreover, there exist two constants $A >0$ and $a>0$ such that  the corresponding heat kernel $\Kernel (t, x, y)$ satisfies the estimate 
\begin{equation} \label{fPS12}
0 < \Kernel (t, x, y) \le A  t^{-\frac {N} {2}} e^{- \frac { |x-y| ^2 } { \cstpa  t}} h(t,x)h(t,y),
\end{equation} 
where 
\begin{equation} \label{fPSeta:2}
h(t,x)= \Bigl( 1+ \frac {\sqrt t} { |x|} \Bigr)^\eta ,
\end{equation} 
and $\eta $ is given by~\eqref{fPSeta}.  
See~\cite[Theorem~1.2]{LiskevichS}, \cite[Theorem~3]{MilmanS}, \cite[Theorem~3.10]{MoschiniT}.

Using the kernel $\Kernel$ we write equation~\eqref{fPT2} in the integral form
\begin{equation} \label{fPT2I} 
w (t) = \int  _{ \R^N  } \Kernel (t, x, y) \DIbd (y)  \xdif y + \int _0 ^t \int  _{ \R^N  }  \Kernel (t-s , x, y)  \Md w(s, y)  \xdif y \xdif s,
\end{equation}
where $\DIbd = u(0)- U(0)$.
Since $\Kernel$ is bounded from below by a term similar to the right-hand side of~\eqref{fPS12}, it follows that $\Kernel (t,x, y)$ has the singularity $ |x|^{- \eta}$ as $ |x|\to 0$ and similarly in $y$. 
Thus we see that the kernel of the operator $e^{tH}$ associated with equation~\eqref{fPT2} is more singular than the heat kernel associated with equation~\eqref{fIntro1}. Of course, the right-hand side of~\eqref{fPT2} is less singular than the right-hand side of~\eqref{fIntro1}.
In fact the worst term in $\Md w$ is of order $( \frac {  |x| } {  |x| +  \sqrt t } )^{\DFNk} |x| ^{ -2 } w$, where $\DFNk >0$ is  given by~\eqref{fDefnKappa} below. (See~\eqref{fPS9b}, \eqref{fPS9:b2}, \eqref{fPS10}.) At positive times, this term is better than $ |x|^{-\frac {2} {\alpha }} w$. However, as $t\to 0$, it behaves like  $ |x|^{-\frac {2} {\alpha }} w$. 
This, combined with the singularity of the kernel, 
excludes the possibility of carrying out a standard contraction mapping argument based on~\eqref{fPT2I}.  
Our solution to this difficulty is taken from~\cite{CDNW2} and involves a contraction mapping argument in a class of functions $w$ that are sufficiently small as $(t,x) \to (0,0)$ so as to balance the singularity of $\Md w$. 
The key point is to find such a class which is preserved by the iterative process. 
The fixed point $w$ thus obtained satisfies the integral equation~\eqref{fPT2I} and in fact solves~\eqref{fPT2} in the sense of distributions. Therefore $u=U+w$ satisfies~\eqref{fPS4:b8} in the sense of distributions. 

We do not know if the condition $\alpha <\alpha _0$ in Theorem~\ref{ePSSZ1} is necessary. 
However, if $\alpha >\alpha _0$ (i.e. $ \beta (\alpha +1) > \frac {(N-2)^2} {4}$), then our proof breaks down from the beginning, since in this case the linear heat equation with potential $ \beta (\alpha +1)  |x|^{-2}$ is ill-posed, see~\cite{BarasG, VazquezZ}. 

The results in this paper are motivated by our article~\cite{CDNW2} where we prove an analogue of Theorem~\ref{ePSSZpr2} where the self-similar solution $U$ has a regular profile. 
Such self-similar solutions have a singularity at $(t,x)= (0,0)$, which introduces some limitations in our results. In particular, we are led to consider in~\cite{CDNW2} initial values $\DIbd $ that equal $U(0, \cdot )$ in a neighborhood of the origin. The same limitation appears here, and we do not know if it is technical or not.

The rest of this paper is organized as follows. 
Section~\ref{sLemma} is devoted to some properties of the parameters we use throughout the paper.
In Section~\ref{sNonHom}, we establish some specific estimates for the nonhomogeneous heat equation with inverse square potential. 
In Section~\ref{sSetting}, we introduce the setting for the fixed point argument that we use  in Section~\ref{sFixedPoint} to prove Theorem~\ref{ePSSZ1}. 
In Section~\ref{sStatSing}, we give a description of the radially symmetric, stationary solutions of~\eqref{NLHE}, showing the existence of sign-changing solutions.   
In Section~\ref{sSingProf}, we deduce Theorem~\ref{ePSSZpr1} and Theorem~\ref{ePSSZpr2} from Theorem~\ref{ePSSZ1}. 
We collect in Appendix~\ref{sISP} some general properties, concerning mostly the linear heat equation with inverse square potential, that we use in the paper and for which we did not find a reference.

\section{Elementary inequalities} \label{sLemma} 
This section is devoted to the following elementary properties.

\begin{lmm} \label{eNUM1} 
Suppose  $N\ge 3$, and let  $\alpha _0> 0$ be defined by~\eqref{fPS0} 
and $\Lambda \in \R$ by
\begin{equation} \label{fPS0:0b} 
\Lambda = \Bigl( \frac {1} {\alpha } - \frac {N-2} {4} \Bigr)^2 -  \Bigl(  \frac {N-2} {2}  - \frac {1} {\alpha } \Bigr) .
\end{equation} 
The following properties hold. 
\begin{enumerate}[{\rm (i)}] 

\item \label{eNUM1:1b1} 
$\alpha _0$ satisfies
\begin{equation} \label{eNUM1:1} 
\frac {2} {N-2}<\alpha _0<\frac {4} {N-2} .
\end{equation} 

\item \label{eNUM1:2} 
For $0 < \alpha < \frac {4} {N-2}$, the three properties
\begin{gather} 
 \alpha  <\alpha _0,   \label{fPShar:b3} \\
\Lambda >0 ,  \label{fPShar:b2}
\end{gather} 
and~\eqref{fPShar} (where $\beta $ is given by~\eqref{bta}), are  equivalent.

\item \label{eNUM1:3} 
Let $0 < \alpha < \alpha _0$ and let $\Lambda >0$ be given by~\eqref{fPS0:0b}.  
If $\mu _1, \mu _2$ are defined by
\begin{gather} 
\mu_ 1=  \frac {1} {\alpha } - \frac {N-2} {4} - \sqrt{ \Lambda  } , \label{fDEFm1}  \\
\mu_ 2=  \frac {1} {\alpha } - \frac {N-2} {4} + \sqrt{  \Lambda } , \label{fDEFm2}
\end{gather} 
then $0 < \mu _1 < \mu _2$.
Moreover, $\frac {(N-2)^2} {4} - \beta  (\alpha +1) >0$ by Property~\eqref{eNUM1:2} above, and if 
$\rho $ is given by~\eqref{fPSrho}, then 
\begin{equation} \label{fPSrho:b1} 
\rho = 2 \mu _1 
\end{equation} 
so that $\rho >0$. 
\end{enumerate} 
\end{lmm} 

\begin{proof} 
Since
\begin{equation*} 
2 \sqrt {N-1} = \sqrt {N^2 - (N-2)^2} 
\end{equation*} 
we see that $2 \sqrt {N-1} <N$. Therefore $N-4+2\sqrt{ N-1 } < 2N-4$, hence $\alpha _0> \frac {2} {N-2}$. 
Moreover, 
\begin{equation*} 
2\sqrt{N-1} > 2, 
\end{equation*} 
so that $N-4+2\sqrt{ N-1 } > N-2$, hence $\alpha _0< \frac {4} {N-2}$. 
This proves Property~\eqref{eNUM1:1}.

We now prove Property~\eqref{eNUM1:2}.
We have
\begin{equation} \label{fFORm1} 
 \frac {(N-2)^2} {4} - \beta  (\alpha +1) =  4  \Lambda 
\end{equation} 
so that~\eqref{fPShar:b2} and~\eqref{fPShar} are equivalent. 
We write
\begin{equation*} 
\frac {1} {\alpha _0}= \frac {N-2} {4} + \frac {\sqrt {N-1} -1} {2}
\end{equation*} 
and 
\begin{equation*} 
\frac {1} {\alpha }= \frac {1} {\alpha _0} + \varepsilon 
\end{equation*} 
so that 
\begin{equation*} 
\Lambda  = \varepsilon ^2 + \varepsilon  \sqrt {N-1} .
\end{equation*} 
It follows that $\Lambda >0$ if and only if either $\varepsilon >0$, i.e. $\alpha < \alpha _0$, or else $\varepsilon < - \sqrt {N-1}$. In this last case, 
\begin{equation*} 
\frac {1} {\alpha }< \frac {1} {\alpha _0} - \sqrt {N-1} = \frac {N-2} {4} - \frac {1} {2} - \frac {\sqrt {N-1}} {2}< \frac {N-2} {4}
\end{equation*} 
so that $\alpha >\frac {4} {N-2}$. Thus we see that~\eqref{fPShar:b3} and~\eqref{fPShar:b2} are equivalent. 

To prove Property~\eqref{eNUM1:3}, we observe that $\mu _2>\mu _1$. Moreover, 
$\Lambda <  ( \frac {1} {\alpha } - \frac {N-2} {4} )^2$.
Thus $ \sqrt \Lambda <  \frac {1} {\alpha } - \frac {N-2} {4} $ so  that $\mu _1>0$.
Finally, formula~\eqref{fFORm1} yields  $\rho =2\mu _1$.
\end{proof} 

\section{The nonhomogeneous heat equation with inverse square potential} \label{sNonHom} 

In this section, we assume~\eqref{fCondAlpha}, and we use the operator $H$ defined by~\eqref{fPShar:b1} and the corresponding semigroup $(e ^{t H}) _{ t\ge 0 }$ with the kernel $\Kernel (t,x,y)$.
We let $\rho , \eta>0$ be defined by~\eqref{fPSrho} and~\eqref{fPSeta}, respectively. 
This section is devoted to estimates of 
\begin{equation*} 
\int _0^t  e^{(t-s)H} f(s)   \xdif s  ,
\end{equation*} 
for some specific right-hand sides $f$.

\begin{lmm} \label{ePSS20}
Let $c\ge 1$, $0\le b\le 2$ and $\kappa \ge 0$ ($\kappa >0$ if $b=2$) satisfy
\begin{equation} \label{fPS20:1:b1} 
 b + (c-1) \eta -\kappa < 2. 
\end{equation}  
Since $2 \eta < N-2$ by~\eqref{fPSeta}, we have $b + (c+1) \eta -\kappa  <N$, and we  fix
\begin{equation}  \label{fSPL2} 
0\le  \varepsilon < \min  \Bigl\{   \frac {1} {4}, N- (b + (1+c) \eta - \kappa) \Bigr\} .
\end{equation} 
Define $ \Psi  (t) $ for $t>0$ by
\begin{equation} \label{fPS20:2}
\Psi   (t,x)= h(t,x)^c   |x|^{-b}\Bigl( \frac { |x|} { |x| + \sqrt t  } \Bigr)^\kappa ,\quad x\in \R^N 
\end{equation} 
where $h$ is given by~\eqref{fPSeta:2}. 
It follows that for every $m>\kappa -2$ there exists $ \CSTB{m} >0$ such that for all $0<r \le \infty $,
\begin{equation} \label{fPS20:3}
\int_0^t e^{(t-s)H} \KIR \Psi   (s, \cdot ) s^{\frac {m} {2}} \xdif s\le \CSTB{m} r^\varepsilon   t^{\frac {m+2-b} {2}}  |x |^{-\varepsilon } h(t, x )  ,
\end{equation} 
for all $t>0$ and $x\in \R^N $, with $\CSTB{m} \to 0$  as $m \to \infty$.
\end{lmm} 

\begin{proof} 
We write using~\eqref{fPS12},
\begin{equation*}
\begin{split} 
e^{(t-s)H}  \KIR \Psi (s) (x)   &  \le A h(t-s, x )   (t -s )^{-\frac {N} {2}}  \int  _{ \R^N  }  e^{- \frac { | x -y| ^2 } { \cstpa  (t-s) }} \KIR h(t-s,y) \Psi (s,y)  \xdif y  \\ & = A (a\pi )^{\frac {N} {2}}  h(t -s , x )  e^{\frac {a} {4} (t-s) \Delta }  [ \KIR h(t-s, \cdot ) \Psi (s,\cdot ) ] (x) ,
\end{split} 
\end{equation*} 
so that
\begin{equation*}
e^{(t-s)H}  \KIR \Psi  (s)  \le  C   h(t -s ,x)  e^{\frac {a} {4} (t-s) \Delta }  \Bigl[ \KIR h(t-s, \cdot ) h(s,\cdot )^c |\cdot  |^{-b} \Bigl( \frac {  | \cdot |} {  | \cdot | + \sqrt s  } \Bigr)^\kappa  \Bigr]   . 
\end{equation*} 
Since
\begin{equation*}
h(t-s, y ) h(s, y )^c  \le C  \Bigl( 1+ \frac { (t-s)^{\frac {\eta } {2}} } { |y|^\eta} \Bigr) \Bigl( 1+ \frac { s ^{\frac {\eta c} {2}} } { |y|^{\eta c}} \Bigr)  \le C  \Bigl( 1+ \frac { (t-s)^{\frac {\eta } {2}} } { |y|^\eta}  + \frac { s ^{\frac {\eta c} {2}} } { |y|^{\eta c}} + \frac { (t-s)^{\frac {\eta } {2}}  s ^{\frac {\eta c} {2}} } { |y|^{(1 + c) \eta}}  \Bigr) ,
\end{equation*} 
we deduce that
\begin{equation} \label{ePSZ4:4}
e^{(t-s)H}  \KIR \Psi  (s) \le  C  h(t -s ,x)  [  I_1 + (t-s)^{\frac {\eta } {2}} I_2 + s ^{\frac {\eta c} {2}} I_3 +  (t-s)^{\frac {\eta } {2}}  s ^{\frac {\eta c} {2}} I_4 ]
\end{equation} 
where
\begin{align*} 
I_1 &=   e^{\frac {a} {4} (t-s) \Delta }  \Bigl( \KIR   |\cdot  |^{-b} \Bigl( \frac {  | \cdot |} {  | \cdot | + \sqrt s  } \Bigr)^\kappa  \Bigr) ,  \\
I_2 &=   e^{\frac {a} {4} (t-s) \Delta }  \Bigl( \KIR   |\cdot  |^{-b- \eta} \Bigl( \frac {  | \cdot |} {  | \cdot | + \sqrt s  } \Bigr)^\kappa  \Bigr) , \\
I_3 &=   e^{\frac {a} {4} (t-s) \Delta }  \Bigl( \KIR  |\cdot  |^{-b - \eta c} \Bigl( \frac {  | \cdot |} {  | \cdot | + \sqrt s  } \Bigr)^\kappa  \Bigr)  , \\
I_4 &=   e^{\frac {a} {4} (t-s) \Delta }  \Bigl( \KIR   |\cdot  |^{-b- (1+ c ) \eta } \Bigl( \frac {   | \cdot |} {  | \cdot | + \sqrt s  } \Bigr)^\kappa  \Bigr) .
\end{align*} 
We note that
\begin{equation} \label{fSPL1} 
\KIR \le r^\varepsilon  |y|^{- \varepsilon } ,
\end{equation} 
and we recall that if $0\le  p <N$, then 
\begin{equation}  \label{ePSZ4:10}
e^{t \Delta }  | \cdot |^{- p } \le C ( t +  |x|^2)^{- \frac {p } {2}} .
\end{equation} 
See e.g.~\cite[Corollary~8.3]{CazenaveDEW}.

Let $\kappa _1 =\min\{\kappa ,b\}$. We have
\begin{equation}  \label{ePSZ4:9}
\begin{split} 
I_1 & \le r^\varepsilon 
e^{\frac {a} {4} (t-s) \Delta }  \Bigl(   |\cdot  |^{-b- \varepsilon } \Bigl( \frac {  | \cdot |} {  | \cdot | + \sqrt s  } \Bigr)^{\kappa _1} \Bigr) 
 \le r^\varepsilon 
 s^{-\frac {\kappa _1} {2}} e^{\frac {a} {4} (t-s) \Delta }  (   |\cdot  |^{- (b +\varepsilon - \kappa _1 ) } )
 \\ & \le C r^\varepsilon  s^{-\frac {\kappa _1} {2}}( t-s +  |x|^2 )^ {-\frac {b + \varepsilon -\kappa _1 } {2}}
 \le C r^\varepsilon  |x|^{- \varepsilon }  s^{-\frac {\kappa _1} {2}}(t-s )^{-\frac {b -\kappa  _1 } {2}} ,
\end{split} 
\end{equation} 
where we used~\eqref{fSPL1}, \eqref{ePSZ4:10} and the property $0\le b + \varepsilon - \kappa _1 \le b +\varepsilon \le \frac {9} {4} < N$.

Similarly, letting $\kappa _2 = \min \{ \kappa ,b + \eta \}$, 
\begin{equation}  \label{ePSZ4:12}
I_2  \le r^\varepsilon   s^{-\frac {\kappa _2 } {2}} e^{\frac {a} {4} (t-s) \Delta }  (   |\cdot  |^{-b-   \eta -\varepsilon  + \kappa _2 } ) 
 \le C r^\varepsilon   s^{-\frac {\kappa _2 } {2}} ( t-s +  |x|^2 )^ {-\frac {b +  \eta +\varepsilon  - \kappa _2 } {2}} 
 \le C r^\varepsilon  |x|^{- \varepsilon }   s^{-\frac {\kappa _2  } {2}} ( t-s  )^ {-\frac {b +  \eta - \kappa _2 } {2}} ,
\end{equation} 
where we used~\eqref{fSPL1}, \eqref{ePSZ4:10} and the property 
\begin{equation*} 
0\le  b +  \eta +\varepsilon  - \kappa _2 \le  b +  \eta +\varepsilon  \le   \frac {9} {4} +  \eta  \le   \frac {9} {4} + \frac {N-2} {2}  <N. 
\end{equation*} 

Next, setting $\kappa _3 = \min \{  \kappa , b+c\eta  \}$,
\begin{equation}  \label{ePSZ4:13}
\begin{split} 
I_3 & \le  r^\varepsilon  s^{-\frac {\kappa _3 } {2}} e^{\frac {a} {4} (t-s) \Delta }  (   |\cdot  |^{-b- c \eta- \varepsilon  + \kappa _3 } ) 
 \le C r^\varepsilon  s^{-\frac {\kappa _3} {2}} ( t-s +  |x|^2 )^ {-\frac {b + c \eta +\varepsilon  - \kappa _3 } {2}} \\
& \le C r^\varepsilon  |x|^{- \varepsilon }  s^{-\frac {\kappa _3} {2}} ( t-s  )^ {-\frac {b + ( c-1) \eta - \kappa  _3} {2}} ( t-s +  |x|^2 )^ {-\frac {  \eta   } {2}} \\
& \le C r^\varepsilon  |x|^{- \varepsilon }  s^{-\frac {\kappa _3} {2}} ( t-s  )^ {-\frac {b + ( c-1) \eta - \kappa  _3} {2}}  \frac {1} { |x|^\eta h(t-s, x)} .
\end{split} 
\end{equation} 
Here we used~\eqref{fSPL1}, \eqref{ePSZ4:10} and the property  $0\le b + c \eta +\varepsilon  - \kappa _3 <N$.
The last inequality is immediate if $\kappa _3= b + c \eta$; and if $\kappa _3 = \kappa $, it follows from $b +  c \eta +\varepsilon  - \kappa  \le b + (1+c) \eta +\varepsilon  - \kappa <N$ by~\eqref{fSPL2}.

Furthermore, setting $\kappa _4 =  \min \{  \kappa , b+ (1 + c) \eta  \}$,
\begin{equation}  \label{ePSZ4:11}
\begin{split} 
I_4 & \le  r^\varepsilon  s^{-\frac {\kappa _4 } {2}} e^{\frac {a} {4} (t-s) \Delta }  (   |\cdot  |^{-b- (1+ c ) \eta- \varepsilon  + \kappa _4} )  \le C r^\varepsilon  s^{-\frac {\kappa _4} {2}} ( t-s +  |x|^2 )^ {-\frac {b + (1+ c ) \eta + \varepsilon  - \kappa _4 } {2}} \\
& \le C r^\varepsilon  |x|^{- \varepsilon }  s^{-\frac {\kappa _4} {2}}  (t-s)^{- \frac {b + c \eta - \kappa _4} {2}} ( t-s +  |x|^2 )^ {-\frac { \eta   } {2}}  \le C r^\varepsilon  |x|^{- \varepsilon }  s^{-\frac {\kappa _4} {2}}  (t-s)^{- \frac {b + c \eta - \kappa _4} {2}}  \frac {1} { |x|^\eta h(t-s, x)} .
\end{split} 
\end{equation} 
In~\eqref{ePSZ4:11} we used~\eqref{fSPL1}, \eqref{ePSZ4:10} and the property  $0\le b + (1+ c) \eta +\varepsilon  - \kappa _4 <N$.
The last inequality is immediate if $\kappa _4= b + (1 + c) \eta$; and if $\kappa _4 = \kappa $, it follows from~\eqref{fSPL2}.
We deduce from~\eqref{ePSZ4:4}, \eqref{ePSZ4:9}, \eqref{ePSZ4:12}, \eqref{ePSZ4:13} and~\eqref{ePSZ4:11} that
\begin{equation*}
\begin{split} 
 r^{ - \varepsilon } |x|^{ \varepsilon }   e^{(t-s)H}  \KIR \Psi  (s)  \le  & C  h(t -s ,x)  [ s^{-\frac {\kappa _1} {2}}(t-s )^{-\frac {b-\kappa _1 } {2}} + s^{-\frac {\kappa _2} {2}}(t-s )^{-\frac {b-\kappa _2 } {2}} ] \\ &  +   C   |x|^{- \eta }  [ s ^{\frac {\eta c -\kappa _3} {2}}  ( t-s  )^ {-\frac {b + (c - 1) \eta - \kappa _3 } {2}} + s ^{\frac {\eta c -\kappa _4 } {2}}  ( t-s  )^ {-\frac {b + (c - 1) \eta - \kappa _4 } {2}} ]  .
\end{split} 
\end{equation*} 
Since $h(t-s, x) \le h(t,x)$ and $ |x|^{- \eta} \le t^{- \frac {\eta } {2}} h(t,x) $, we deduce that
\begin{equation*}
\begin{split} 
 r^{ - \varepsilon } |x|^{ \varepsilon }   e^{(t-s)H}  \KIR \Psi  (s)  \le  & C  h(t  ,x)  \bigl[ s^{-\frac {\kappa _1 } {2}}(t-s )^{-\frac {b-\kappa  _1 } {2}}  +s^{-\frac {\kappa _2 } {2}}(t-s )^{-\frac {b-\kappa  _2 } {2}}   \\ & +  t^{- \frac {\eta } {2}}  \bigl(  s ^{\frac {\eta c -\kappa _3} {2}}  ( t-s  )^ {-\frac {b + (c - 1) \eta - \kappa _3 } {2}} +  s ^{\frac {\eta c -\kappa _4} {2}}  ( t-s  )^ {-\frac {b + (c - 1) \eta - \kappa _4 } {2}} \bigr) \bigr] .
\end{split} 
\end{equation*} 
Multiplying by $s^{ \frac {m } {2}} $ and integrating in $s\in (0,t)$, we obtain~\eqref{fPS20:3} with
\begin{equation} \label{ePSZ4:15}
B_m =  C  \sum_{ j=1 }^2  \int _0^1 \sigma ^{ \frac {m - \kappa _j } {2}}( 1- \sigma  )^{-\frac {b-\kappa  _j } {2}} \xdif \sigma     + 
C \sum_{ j=3 }^4 \int _0^1 \sigma  ^{\frac { m + \eta c -\kappa _j} {2}}  (1- \sigma   )^ {-\frac {b + (c - 1) \eta - \kappa _j } {2}}  \xdif \sigma .
\end{equation} 
The integrals in~\eqref{ePSZ4:15} are finite. Indeed, for all $j\in \{1, 2, 3, 4\}$,
\begin{equation*} 
\min  \Bigl\{ \frac {m - \kappa _j } {2}, \frac { m + \eta c -\kappa _j } {2} \Bigr\} \ge \frac { m  -\kappa } {2} > -1.
\end{equation*} 
Moreover,
$\frac {b-\kappa  _j  } {2} < \frac {b } {2} \le 1$ for $j=1,2$. Furthermore, if $\kappa _j= \kappa $ for $j=3$ or $j=4$, then
\begin{equation*} 
\frac {b + (c - 1) \eta - \kappa _j } {2}= \frac {b + (c - 1) \eta - \kappa  } {2} <1
\end{equation*} 
by~\eqref{fPS20:1:b1}.
If $\kappa _3= b+ c\eta$, then $\frac {b + (c - 1) \eta - \kappa _3 } {2}= - \frac { \eta  } {2} < 0<1$. Similarly if 
$\kappa _4= b+ (1 + c) \eta$, then $\frac {b + (c - 1) \eta - \kappa _4 } {2}= -\eta< 0<1$.
Finally, the property $B_m \to 0$ follows by dominated convergence.
\end{proof}

\begin{lmm} \label{ePSbck} 
Let 
\begin{align*} 
(b_1, c_1, \kappa _1) &= (2, 1, \rho ), \\ (b_2, c_2, \kappa _2) &= (0, 1+\alpha , 0 ), \\ (b_3, c_3, \kappa _3) &= ( {\textstyle { \frac {2(\alpha -1)}  {\alpha }}}, 2, 0) , \\  (b_4, c_4, \kappa _4) &= (2, 1, \alpha \rho ). 
\end{align*} 
It follows that $b_3<2$ and that the triplets $(b_j, c_j, \kappa _j)$ satisfy~\eqref{fPS20:1:b1}
for $j=1,2,3,4$. 
\end{lmm} 

\begin{proof} 
We first note that by~\eqref{fPSeta} 
\begin{equation} \label{fDDDe1} 
\eta < \frac {N-2} {2} < \frac {2} {\alpha },
\end{equation} 
since $\alpha <\frac {4} {N-2}$.
Next,
\begin{equation*} 
 b_1 + (c_1 -1) \eta -\kappa _1 = 2 - \rho <2.
\end{equation*} 
For $j=2$,
\begin{equation*} 
 b_2 + (c_2 -1) \eta -\kappa _2 = \alpha \eta < 2 
\end{equation*}
by~\eqref{fDDDe1}. 
For $j=3$,  
\begin{equation*} 
 b_3 + (c_3 -1) \eta -\kappa _3 = 2 - \frac {2} {\alpha } + \eta < 2
\end{equation*}
by~\eqref{fDDDe1}. 
For $j=4$,
\begin{equation*} 
 b_4 + (c_4 -1) \eta -\kappa _4 = 2- \alpha \rho <2.
\end{equation*} 
This completes the proof.
\end{proof} 

\begin{lmm} \label{ePSRM} 
Let 
\begin{equation} \label{fDefnKappa} 
\DFNk = \rho \min \{ 1 ,\alpha \}  
\end{equation} 
and let 
\begin{equation} \label{fPSGS} 
g(t,x)= 
  \Bigl( \frac {  |x| } {  |x| + \sqrt t  } \Bigr)^{\DFNk  }  |x|^{-2} + (\CSTK h(t,x))^{\alpha } + \CSTK \widetilde{ g} (t,x)
\end{equation} 
for a.a. $t>0$, $x\in \R^N $, where $\CSTK \ge 0$, $h$ is given by~\eqref{fPSeta:2}, and
\begin{equation*} 
 \widetilde{ g} (t,x) = 
\begin{cases} 
0 & \alpha \le 1 \\
 |x|^{-\frac {2(\alpha -1)} {\alpha }} h(t,x) & \alpha >1 .
\end{cases} 
\end{equation*} 
It follows that there exists a constant $ \CSTu $ such that
\begin{align} 
 g(t, x) &\le \CSTu  ( 1 +  |x|^{-2} ) \label{fPSRM:b1} \\
  g(t, x) h(s ,x) &\le \CSTu  ( 1 +  |x|^{-\frac {N+2} {2}} ) \label{fPSRM:b2} \\
    g(t, x) h(s ,x) h(\sigma ,x)  &\le \CSTu   ( 1 +  |x|^{-N + \vartheta } ) \label{fPSRM:b3} 
\end{align} 
for a.a. $t,s,\sigma \in (0,1 ) $ and $x\in \R^N $, where 
\begin{equation} \label{fDefndelta} 
\vartheta = \frac {N-2} {2} - \eta >0.
\end{equation} 
Moreover, there exists $\varepsilon _0>0$ such that if $0\le \varepsilon < \varepsilon _0$ is fixed, then for every $m>  \max \{ \rho ,\alpha \rho \}  -2$, 
\begin{equation} \label{fPSRM} 
\int_0^t e^{(t-s)H} \KIR g(s) h(s)s^{\frac {m} {2}}  \xdif s  \le \CSTR{m} r^{\varepsilon }  t^{\frac {m} {2}}  |x|^{-\varepsilon } h(t,x)
\end{equation} 
where 
\begin{equation}  \label{fPSDZ16} 
 \CSTR{m}  \goto_{m\to \infty} 0 .
\end{equation}   
\end{lmm} 
\begin{proof} 
We first prove~\eqref{fPSRM:b1}-\eqref{fPSRM:b3}.  
It follows from~\eqref{fPSeta:2} and~\eqref{fPSGS} that for $0\le t\le 1$
\begin{align} 
 g(t, x) &\le C( 1 +  |x|^{-2} +  |x|^{-\alpha \eta} +  |x|^{-2 + \frac {2} {\alpha } -\eta}  ) \label{fPSRM:b4}  \\
   g(t, x) h(s,x) &\le C( 1 +  |x|^{-2- \eta} +  |x|^{- (\alpha +1) \eta} +  |x|^{-2 + \frac {2} {\alpha } - 2\eta}  ) \label{fPSRM:b5} \\
      g(t, x) h(s,x) h (\sigma , x) &\le C( 1 +  |x|^{-2- 2 \eta} +  |x|^{- (\alpha +2) \eta} +  |x|^{-2 + \frac {2} {\alpha } - 3\eta}  ) \label{fPSRM:b6} 
\end{align} 
Since 
\begin{equation} \label{fEAE1} 
\alpha <\frac {4} {N-2} \quad  \text{and} \quad \eta =  \frac {N-2} {2}- \vartheta  < \frac {N-2} {2},
\end{equation} 
we see that $\eta \alpha <2$ and $ \frac {2} {\alpha } -\eta>0$, so that~\eqref{fPSRM:b1} follows from~\eqref{fPSRM:b4}. 
Using again~\eqref{fEAE1}, we obtain $2+ \eta < \frac {N+2} {2}$, $(\alpha +1) \eta < \frac {N+2} {2}$ and $ 2 -\frac {2} {\alpha } +2 \eta < \frac {N+2} {2}$, hence~\eqref{fPSRM:b2} follows from~\eqref{fPSRM:b5}. 
Moreover, \eqref{fEAE1} yields
\begin{gather*} 
2+ 2\eta = N- 2 \vartheta \le  N-\vartheta ,\\ (\alpha +2) \eta \le N- \vartheta \frac {2N} {N-2} \le N-\vartheta ,\\ 
2 - \frac {2} {\alpha } + 3\eta \le  N-3\vartheta \le N-\vartheta ,
\end{gather*} 
so that~\eqref{fPSRM:b3} follows from~\eqref{fPSRM:b6}. 

Estimate~\eqref{fPSRM} follows from Lemma~\ref{ePSS20}.
For the term $  ( \frac {  |x| } {  |x| + \sqrt t  } )^{\DFNk  }  |x|^{-2}  $ we apply Lemma~\ref{ePSS20} with $(b, c , \kappa )= (2, 1, \rho )$ if $\DFNk =\rho $ and  $(b, c , \kappa )= (2, 1, \alpha \rho )$ if $\DFNk =\alpha \rho $. This is possible by Lemma~\ref{ePSbck}. 
For the term $ h^\alpha  $ we apply Lemma~\ref{ePSS20} with $(b, c , \kappa )= (0, 1+\alpha , 0)$. This is again possible by Lemma~\ref{ePSbck}. 
Finally for the term $ \widetilde{g} $, we need only consider the case $\alpha >1$ and we apply Lemma~\ref{ePSS20} with $(b, c , \kappa )= ( \frac {2 (\alpha -1)} {\alpha }, 2 , 0)$, which is possible by Lemma~\ref{ePSbck}. 
\end{proof} 

\section{The setting for the fixed-point argument} \label{sSetting} 

In this section, we assume~\eqref{fCondAlpha}, and we use the operator $H$ defined by~\eqref{fPShar:b1} and the corresponding semigroup $(e ^{t H}) _{ t\ge 0 }$ with the kernel $\Kernel (t,x,y)$.
We introduce the framework for the fixed-point  argument that  we use for the proof of Theorem~\ref{ePSSZ1}.

We begin with the definition of several auxiliary functions.
Let $\delta >0$, and set  
\begin{equation} \label{fPSchi:b1} 
a_j= 2^{-j} \delta 
\end{equation} 
for $j\ge 0$. Define the sequence $(\chi _j) _{ j\ge 0 } \subset L^\infty  (\R^N ) $ by
\begin{equation}  \label{fPSchi} 
\chi _j (x) =
\begin{cases} 
0 &  |x| \le  a_j \\ 1 &  |x|> a_j .
\end{cases} 
\end{equation}
Given $T>0$ and an integer $m \ge 1$, we set  
\begin{equation} \label{fPStheta} 
\Theta (t, x) =  h(t,x) \Bigl(  t^{\frac {m} {2}} + \sum_{ j=1 }^m t^{\frac {j-1} {2}} \chi _j   \Bigr),
\end{equation} 
for $0\le t\le T$ and $x\in \R^N $, where $h$ is given by~\eqref{fPSeta:2}. 
Given $\CSTK >0$, we define
\begin{equation} \label{fPSDZ14} 
\Ens = \{w\in \xLn{1} _\Loc ((0,T) \times \R^N ) ; \,    |w| \le \CSTK \Theta  \} ,
\end{equation} 
and
\begin{equation} \label{fPSDZ14:b1} 
\dist (w, z) =   \Bigl\| \frac {w-z } {\Theta}  \Bigr\| _{ \xLn{\infty }  ((0,T)\times \R^N  ) }, \quad w,z\in \Ens ,
\end{equation} 
so that $(\Ens, \dist )$ is a complete metric space.

\begin{lmm} \label{ePSS2} 
 Let $\delta >0$. With the notation~\eqref{fPSchi:b1}-\eqref{fPSchi}, it follows  that there exists $\CSTtu >0$ such that for all  $j\ge 0$ and all $0\le t,s\le 1$
\begin{equation} \label{fPSD} 
e^{t H} \chi _j  \le \CSTtu   h(t,x) \Bigl( e^{-\frac { a _{ j+1 }^2} {2  \cstpa  t}}  +  \chi  _{ j+1 }  \Bigr)
\end{equation} 
where $ \cstpa $ is the constant in~\eqref{fPS12} and $h$ is given by~\eqref{fPSeta:2}.
\end{lmm} 

\begin{proof} 
Applying~\eqref{fPS12},
\begin{equation} \label{fPES10} 
 \int  _{ \R^N  } \Kernel (t, x, y)  \chi _j (y)  \xdif y  \le C t ^{-\frac {N} {2}}  h(t,x)  \int  _{ \{  |y|> a_j \}  } e^{-\frac { |x-y|^2} { \cstpa t}}  h (t,y) \xdif y .
\end{equation} 
We also write
\begin{equation} \label{fPES10:b1} 
 \int  _{ \R^N  } \Kernel (t, x, y)  \chi _j (y)  \xdif y  \le C t ^{-\frac {N} {2}}  h(t,x)   \int  _{ \R^N   } e^{-\frac { |y|^2} { \cstpa t}} h (t,x-y)    \xdif y .
\end{equation} 
We deduce from~\eqref{fPES10:b1} that 
\begin{equation*} 
\begin{split} 
 \int  _{ \R^N  } \Kernel (t, x, y)  \chi _j (y)  \xdif y  &  \le C  t ^{-\frac {N} {2}} h(t,x)    \int  _{   \R^N  } e^{-\frac { |y|^2} { \cstpa t}}  \Bigr (1+\frac {\sqrt{ t } } {|x-y|}\Bigr )^{ \eta }  \xdif y \\  &  = C  h(t,x)   \int  _{ \R^N  }e^{-\frac { |y|^2} { \cstpa }}  \Bigr (1+\frac {1 } {| ( x / \sqrt{ t }) -y|}\Bigr )^{\eta }   \xdif y. 
\end{split} 
\end{equation*} 
Since  $\eta <N$ by~\eqref{fPSeta}, we see that
\begin{equation*} 
 \int  _{ \R^N  } e^{-\frac { |y|^2} { \cstpa }}  \Bigr (1+\frac {1 } {| z -y|}\Bigr )^{\eta } \le 2^\eta \int   _{ \R^N  }  e^{-\frac { |y|^2} { \cstpa }} + 2^\eta \int   _{  |z-y|<1 }  | z -y|^{ - \eta } \le C
\end{equation*} 
independent of $z\in \R^N $, and it follows that
\begin{equation*} 
 \int  _{ \R^N  } \Kernel (t, x, y)  \chi _j (y)   \xdif y   \le  C h(t,x) 
\end{equation*} 
for all $x\in \R^N $. 
In particular, if $ |x| > a _{ j+ 1 }$, then
\begin{equation} \label{fPES10:b3} 
 \int  _{ \R^N  } \Kernel (t, x, y)  \chi _j (y)   \xdif y  \le  C h(t,x) \chi _{ j+1 } (x).
\end{equation} 
Next, if $ |x| \le a _{ j+ 1 }$ and $ |y|\ge a_j$, then 
\begin{equation*} 
 |x-y| \ge  |y| - a _{ j+1 } =  |y| - \frac {1} {2} a_j \ge \frac {1} {2}  |y| \ge a _{ j+1 }
\end{equation*} 
and so 
\begin{equation} \label{fPES10b}
e^{-\frac { |x-y|^2} { \cstpa t}} = e^{-\frac { |x-y|^2} {2 \cstpa t}} e^{-\frac { |x-y|^2} {2 \cstpa t}}\le  e^{-\frac { |y|^2} {8 \cstpa t}} e^{-\frac { a _{ j+1 }^2} { 2 \cstpa t}}. 
\end{equation} 
If $ |x| \le a _{ j+ 1 }$, then we deduce from~\eqref{fPES10}, \eqref{fPES10b}, and $h(t, \cdot ) \le h(1, \cdot )$ that 
\begin{equation} \label{fPES10:b2} 
 \int  _{ \R^N  } \Kernel (t, x, y)  \chi _j (y)   \xdif y   \le C h(t,x)  e^{-\frac { a _{ j+1 }^2} { 2 \cstpa t}}  t ^{-\frac {N} {2}} \int  _{ \R^N  } e^{-\frac { |y|^2} { 8 \cstpa t}}  h ( 1 ,y)  \xdif y  \le C   h(t,x)    e^{-\frac { a _{ j+1 }^2} { 2 \cstpa t}} .
\end{equation} 
Then~\eqref{fPSD} follows from~\eqref{fPES10:b3} and \eqref{fPES10:b2}.
\end{proof} 

\begin{lmm} \label{ePSS4} 
Let $S>0$ and let $U \in  \xLn{1}  ((0, S ) \times \R^N )+ \xLn{\infty } ((0, S ) \times \R^N )$ satisfy~\eqref{fPShrho}   
for a.a. $0<t<S$ and $x\not = 0$, where $\rho $ is given by~\eqref{fPSrho}. Let $T\le S$,
\begin{equation} \label{fTUQ} 
0 < T < \frac {1} {4} ,
\end{equation} 
$\delta >0$, $\CSTK >0$, and let $\Ens$ be defined by~\eqref{fPSDZ14}. 
It follows that $\Md w \in \xLn{1} _\Loc ( (0,T)\times \R^N )$ for all $w\in \Ens $, where $\Md w$ is defined by~\eqref{fPT3}.
Moreover, 
\begin{equation} \label{fpPSS4:1} 
  | \Md w (t, \cdot )- \Md z (t, \cdot )| \le   \CSTtd |w(t, \cdot )-z(t, \cdot )| g(t, \cdot ),
\end{equation} 
for all $w, z \in \Ens$, where  $g(t,x) $ is given by~\eqref{fPSGS} and $ \CSTtd $ is independent of $T$, $m$, $\CSTK $, $w$ and $z$.
\end{lmm}

\begin{proof}
Let $w, z \in \Ens$. Set 
\begin{equation*} 
f( s )=  |s|^\alpha s
\end{equation*} 
and define 
\begin{equation}  \label{fPS9b}
Z( x )= \beta ^{\frac {1} {\alpha }}   |x| ^{-\frac {2} {\alpha }}, \quad V= U- Z .
\end{equation} 
It follows that
\begin{equation*}
\Md w=  f (Z + V +w) -   f (Z + V)   - f '(  Z )  w ,
\end{equation*}
so that
\begin{equation} \label{fPS9}
\begin{split}
\Md w  -\Md z&= f (Z + V +w) -   f (Z + V+z)   - f '(  Z ) (w-z)\\
&= (w-z)  \int _0^1 [ f '( Z + V+ s w+(1-s)z) - f '(  Z )  ]   \xdif s .
\end{split}
\end{equation} 
Since $f '(s) = (\alpha +1)  |s|^\alpha $ and 
\begin{equation*}
 |\,  | x |^\alpha - | y |^\alpha  |\le
 \begin{cases}
 \alpha ( |x |^{\alpha -1} +  | y |^{\alpha -1})  | x-y |
 & \text{if }\alpha \ge 1, \\
  | x - y|^\alpha &  \text{if }0<\alpha \le 1,
 \end{cases}
\end{equation*}
for all $x,y\in \R$, we deduce from~\eqref{fPS9} that
\begin{equation} \label{fPS9:b2}
 |\Md w -  \Md z| \le C | w - z|   
 \times  \begin{cases} 
   |V|^{ \alpha } + |w|^{ \alpha } + |z|^{ \alpha } + Z^{ \alpha -1}  ( |V| + | w| + |z| ) &  \text{if }  \alpha \ge 1, \\
  | V|^\alpha +  | w|^\alpha + |z|^{ \alpha }&  \text{if }  0<\alpha \le 1 .
 \end{cases} 
\end{equation} 
By~\eqref{fPShrho} 
\begin{equation}  \label{fPS10}
 | V(t,  x )   | \le C  \Bigl [\Bigl( \frac {  |x| } {  |x| +  \sqrt t } \Bigr)^\rho  |x| ^{-\frac {2}  {\alpha }} +1\Bigr ] .
\end{equation} 
On the other hand, ~\eqref{fPStheta} and~\eqref{fTUQ} yield 
\begin{equation} \label{fET1} 
|w| +  |z| \le  2 \CSTK  \Theta (t,x)  \le 2\CSTK h(t,x)\sum_{ j=0 }^{m} t^{\frac {j} {2}} < 4\CSTK h(t,x).
\end{equation} 
if $t< T$. 
From \eqref{fPS10}  and \eqref{fET1} we get  (recall that $h\ge 1$)
\begin{equation} \label{fET2}
   |V|^{ \alpha } + |w|^{ \alpha } + |z|^{ \alpha } \le  C\Bigl (\Bigl( \frac {  |x| } {  |x| + \sqrt t  } \Bigr)^{\alpha \rho }  |x| ^{-2}+(\CSTK  h(t,x))^{\alpha }\Bigr )
\end{equation}    
and
\begin{equation} \label{fET3}
 Z^{ \alpha -1}  ( |V| + | w| +  |z|)\le C\Bigl (\Bigl( \frac { |x|} {  |x| + \sqrt t  } \Bigr)^{\rho }  |x| ^{-2}+r^{-\frac {2(\alpha -1)} {\alpha }}\CSTK h(t,x)\Bigr ).
 \end{equation} 
Estimate~\eqref{fpPSS4:1} follows from \eqref{fPS9:b2}, \eqref{fET2}, \eqref{fET3} and~\eqref{fPSGS}. 
\end{proof} 

\begin{lmm} \label{eLemSol1} 
Let $\delta >0$,  $T>0$, and $m>  \max \{ \rho ,\alpha \rho \}  -2$, $m\ge 1$.
Let $U \in  \xLn{1}  ((0, T ) \times \R^N )+ \xLn{\infty }  ((0, T ) \times \R^N )$ satisfy~\eqref{fPShrho}   
for a.a. $0<t<T$ and $x\not = 0$.
Let  $\chi _0$ defined by~\eqref{fPSchi} and $\Theta $  by~\eqref{fPStheta}. 
Suppose $ \DIbd \in L^\infty  (\R^N ) $ satisfies $ | \DIbd | \le C \chi _0$ and $w\in \xLn{1} _\Loc ((0, T) \times \R^N )$ satisfies $ |w| \le C \Theta  $ for some constant $C$. It follows that
\begin{equation} \label{eLemSol1:2} 
W(t,x)=: e^{t H } \DIbd + \int _0^t   e^{ (t-s) H} \Md w (s)  \xdif s,
\end{equation} 
where $\Md $ is given by~\eqref{fPT3}, is well defined and $  |W |\le  \widetilde{C} h$ with $ \widetilde{C} >0$ and $h$ given by~\eqref{fPSeta:2}. In particular, $W\in \xLn{{ \frac {2N} {N-2} }} ((0,T) \times \R^N ) + \xLn{\infty }  ((0, T) \times \R^N ) $. 

Moreover, $|x|^{-2} W \in \xLn{1} _\Loc ((0, T) \times \R^N )$ and
\begin{equation} \label{feLemSol1:1} 
\partial _t W - \Delta W - \beta (\alpha +1)  |x|^{-2} W = \Md w 
\end{equation} 
in ${\mathcal D}' ((0, T) \times   \R^N )$. In addition, $W, \Md w \in \xLn{\infty }   ( (0, T) \times \{  |x| >\varepsilon \} ) $ for every $\varepsilon >0$. 
\end{lmm} 

\begin{proof} 
The contribution of $\DIbd $ is estimated by~\eqref{fPSD}.  
Next, using~\eqref{fpPSS4:1} with $z=0$, and the inequality $\Theta (s,y) \le s^{\frac {m} {2}} + C ' \chi _m (y) $, we have
\begin{equation} \label{fSPL3} 
\begin{split} 
 |\Md w ( s, y) | & \le   \CSTtd  g(s,y)  |w (s, y) | \le C  \CSTtd  g(s, y)  \Theta (s, y) \\ & \le C  \CSTtd  s^{\frac {m} {2}} g(s, y)  + C  \CSTtd  C' \chi _m (y) g(s,y) \\
 &  \le C''  s^{\frac {m} {2}} g(s, y)  + C'' \chi _m (y)  ,
\end{split} 
\end{equation} 
for some constants $C', C''$. 
(In the last inequality, we used the fact that $g$ defined by~\eqref{fPSGS} is bounded  on the support of $\chi _m$.)
Since $ h\ge 1$, the contribution of $s^{\frac {m} {2}} g(s, y)$ is estimated by~\eqref{fPSRM} with $\varepsilon =0$ and $r=\infty $; and the contribution of $\chi _m$ is estimated by~\eqref{fPSD} and the inequality $h(t-s, x) \le h(t,x)$.
It follows that $ |W| \le C h$, and the $\xLn{p} $ regularity of $W$ follows from $h (t,x) \le C + C  |x|^{- \eta }$ and $\eta < \frac {N-2} {2}$. 
The same estimate also shows that $|x|^{-2} W \in \xLn{1} _\Loc ((0, T) \times \R^N )$.
That $W, \Md w \in \xLn{\infty }  ( (0, T) \times \{  |x| >\varepsilon \} ) $ for every $\varepsilon >0$ follows from $ |W|\le Ch$ and from~\eqref{fpPSS4:1} with $z=0$. 

We now prove~\eqref{feLemSol1:1}. 
The term corresponding to $\DIbd $ in~\eqref{eLemSol1:2} satisfies the homogeneous equation in ${\mathcal D}' ((0, T) \times  \R^N )$ by Lemma~\ref{eLemLHISP4}, so we now assume $\DIbd =0$. 
We fix a function
\begin{equation*} 
\zeta \in \xCn{\infty } _\Comp ((0, T) \times \R^N  ).
\end{equation*} 
Next, let $\xi \in \xCn{\infty } _\Comp (\R^N ) $ satisfy $0\le \xi \le 1$,  $\xi (x) = 1$ for $ |x|\le 1$ and $\xi (x) = 0 $ for $ |x|\ge 2$, and set 
\begin{equation*} 
\psi  _n ( x) = \xi (nx), \quad \theta  _n (x) = 1 - \xi  \Bigl( \frac {x} {n} \Bigr), \quad \rho _n (x)= 1- \psi _n (x) - \theta _n (x),
\end{equation*} 
for $x\in \R^N $ and $ n\ge 1$. 
It follows that 
\begin{equation} \label{feLemSol1:2} 
 \| \psi _n\| _{ L^r (\R^N ) } \goto  _{ n\to \infty  } 0
\end{equation} 
for all $1\le r<\infty $, that $\rho _n  $ is supported in $\{ \frac {1} {n} \le  |x| \le 2n \}$, and that $0\le \theta _n \le 1$ and $\theta _n (x) =0$ for $ |x|\le n$. 
We write
\begin{equation*} 
W= V_n + W_n + Z_n,
\end{equation*} 
where
\begin{gather*} 
V_n = \int _0^t e^{ (t-s) H} (\psi _n \Md w) (s,\cdot )  \xdif s, \\
W_n = \int _0^t e^{ (t-s) H} (\rho _n \Md w) (s,\cdot )  \xdif s, \\
Z_n = \int _0^t e^{ (t-s) H} (\theta _n \Md w) (s,\cdot )  \xdif s. 
\end{gather*} 
Since $ |\Md w | \le C g \Theta $ and $\Theta \le Ch $, it follows from~\eqref{fPSRM:b2} that
\begin{equation} \label{fEstMdw} 
 | \Md w |\le C ( 1 +  |x|^{-\frac {N+2} {2}} ) .
\end{equation} 
In particular, we see that $\rho _n \Md w \in \xLn{\infty } ((0, T), L^2 (\R^N ) )$.  
Applying Lemma~\ref{eLemLHISP5}, we deduce that
\begin{equation*} 
\int _0^T\int  _{ \R^N  } W_n (- \partial _t \zeta - \Delta \zeta - \beta (\alpha +1)  |x|^{-2}  \zeta )=  \int _0^T\int  _{ \R^N  } \zeta \rho _n \Md w .
\end{equation*} 
Since $N\ge 3$, the right-hand side of~\eqref{fEstMdw} is in $\xLn{1} _\Loc (\R^N )$, so that $\zeta \Md w \in \xLn{1} ((0,T) \times \R^N )$. Since $0\le \rho _n \le 1$ and $\rho _n \to 1$ a.e., we deduce by dominated convergence that
\begin{equation} \label{feLemSol1:3} 
\int _0^T\int  _{ \R^N  } W_n (- \partial _t \zeta - \Delta \zeta - \beta (\alpha +1)  |x|^{-2}  \zeta ) \goto _{ n\to \infty  }  \int _0^T\int  _{ \R^N  } \zeta  \Md w .
\end{equation} 
Next, we let $0<r < 2^{-m} \delta $, so that $\KIR \chi _m \equiv 0$ by~\eqref{fPSchi}; and so it follows from~\eqref{fSPL3} that 
\begin{equation*} 
\KIR  |\Md w ( s, y) | \le C''  s^{\frac {m} {2}} \KIR g(s, y) .
\end{equation*} 
Since $\psi _n \le 1 _{ \{  |y|< \frac {2} {n}\} }$ and $h\ge 1$, we deduce that for $n\ge \frac {2} {r}$
\begin{equation} \label{fSPL4} 
\psi _n  |\Md w ( s, y) | \le C''  s^{\frac {m} {2}} 1 _{ \{  |y|< \frac {2} {n}\} } g(s, y) \le C''  s^{\frac {m} {2}} 1 _{ \{  |y|< \frac {2} {n}\} } g(s, y) h(s, y) . 
\end{equation} 
We fix $0<\varepsilon \le \varepsilon _0$, where $\varepsilon _0$ is given by Lemma~\ref{ePSRM}, sufficiently small so that
\begin{equation} \label{fSPL5} 
2 + \eta + \varepsilon <N.
\end{equation} 
(This is possible, since $\eta < \frac {N-2} {2}$.)
It follows from~\eqref{fSPL4}, \eqref{fPSRM} and~\eqref{fPSeta:2} that
\begin{equation*} 
V_n (t,x) \le C n^{- \varepsilon }   |x|^{-\varepsilon } h(t,x) \le C n^{- \varepsilon }   |x|^{- \eta -\varepsilon },
\end{equation*} 
on the support of $\zeta $, and we conclude using~\eqref{fSPL5}  that
\begin{equation} \label{feLemSol1:5} 
\int _0^T\int  _{ \R^N  } V_n (- \partial _t \zeta - \Delta \zeta - \beta (\alpha +1)  |x|^{-2}  \zeta ) \goto  _{ n\to \infty  } 0.
\end{equation} 
Moreover, $0\le \theta _n \le  1 _{ \{  |x| > n\} } $, so that by Lemma~\ref{eLemLHISP2}
\begin{equation*} 
\sup  _{ 0\le t\le T }  \Bigl\|  \frac {1} {h(T)} Z_n (t) \Bigr\| _{ L^\infty (  |x| \le \frac {n} {2}) } \le C \int _0^T  \Bigl(  1 + \frac {\sqrt s} {n} \Bigr)^\eta e^{-  \varsigma \frac {n^2} {s} }  \xdif s \goto _{ n\to \infty  } 0. 
\end{equation*} 
Since $ \frac {1} {h(T)} \ge \varepsilon  >0$ on the support of $\zeta $, and since  the support of $\zeta $ is included in $\{ |x| \le \frac {n} {2})\}$ for $n$ large, we deduce that
\begin{equation} \label{feLemSol1:6} 
\int _0^T\int  _{ \R^N  } Z_n (- \partial _t \zeta - \Delta \zeta - \beta (\alpha +1)  |x|^{-2}  \zeta ) \goto _{ n\to \infty  }0. 
\end{equation} 
Applying~\eqref{feLemSol1:3}, \eqref{feLemSol1:5} and~\eqref{feLemSol1:6}, we see that
\begin{equation*} 
\int _0^T\int  _{ \R^N  } W (- \partial _t \zeta - \Delta \zeta - \beta (\alpha +1)  |x|^{-2}  \zeta )=  \int _0^T\int  _{ \R^N  } \zeta  \Md w ,
\end{equation*} 
Since $\zeta \in \xCn{\infty }  _\Comp ((0, T) \times \R^N  ) $ is arbitrary, this proves that $W$ is a solution of~\eqref{feLemSol1:1} in ${\mathcal D}' ((0, T) \times  \R^N  )$. 
\end{proof} 

\section{Proof of Theorem~$\ref{ePSSZ1}$}
\label{sFixedPoint} 

We  prove Theorem~\ref{ePSSZ1} by using a fixed point argument.
We set 
\begin{equation*} 
\DIbd = \DI - U(0),
\end{equation*} 
so that
\begin{equation} \label{ePSSZ1:5} 
\DIbd \in \xLn{\infty } (\R^N ) \quad  \text{and}\quad  \DIbd =0  \text{ on } \{  |x|<\delta  \} .
\end{equation} 
We let
\begin{equation} \label{fDK1} 
\CSTK = 2  \CSTtu    \| \DIbd \| _{ \xLn{\infty }  },
\end{equation} 
where $ \CSTtu  $ is given by~\eqref{fPSD}. 
We fix an integer $m\ge 2$, $m>  \max \{ \rho ,\alpha \rho \}  -2$ (where $\rho>0$ is defined by~\eqref{fPSrho}), sufficiently large so that
\begin{equation} \label{fDRM} 
 \CSTR{m} <\frac {1} {4   \CSTtd  },
\end{equation} 
where $  \CSTtd  $ is given by~\eqref{fpPSS4:1} and $ \CSTR{m} $   is given by~\eqref{fPSRM} with $\varepsilon =0$. 
Next, we fix $T\le S$, 
\begin{equation} \label{fPST} 
0< T<\frac {1} {4}
\end{equation} 
sufficiently small so that
\begin{gather} 
e^{-\frac {a_{m+ 1} ^{2}}{2 \cstpa t}} \le  t^{\frac {m + 1 } {2} }\quad  \text{for }0<t\le T , \label{fDFM:1}  \\
T^{\frac {1} {2}} \CSTtu   \CSTtd  \CSTu  \bigl(1+  a_m^{ - \frac {N+2} {2} } \bigr) \le \frac {1} {8 } \label{fDFM:4}  
\end{gather} 
where   $\CSTu $ is given by~Lemma~\ref{ePSRM} and the numbers $a_j$ are given by~\eqref{fPSchi:b1}.
Let $\Theta $ be given by~\eqref{fPStheta} and let $(\Ens, \dist )$ be defined by~\eqref{fPSDZ14}-\eqref{fPSDZ14:b1}  with $\CSTK $ given by~\eqref{fDK1}. 
We define $\Phi : \Ens \mapsto \xLn{2} ((0,T) \times \R^N )+ \xLn{\infty }  ((0,T) \times \R^N )$ by
\begin{equation} \label{fNVDP1} 
 \Phi (w) (t) = e ^{t H} \DIbd  +  \int _0^t  e^{ (t-s) H } \Md w (s) \xdif s .
\end{equation} 
(Recall that $ \Phi (w) \in \xLn{{ \frac {2N} {N-2} }} ((0,T) \times \R^N ) + \xLn{\infty } ((0, T) \times \R^N ) $ is well defined by Lemma~\ref{eLemSol1}.)
We will show that $\Phi(\Ens)\subset \Ens$ and that $\Phi$ has a unique fixed point in $\Ens$.

Using~\eqref{ePSSZ1:5}, we write  $ | \DIbd  |  \le  \| \DIbd  \| _{ \xLn{\infty }  } \chi _0 $, and we deduce from~\eqref{fPSD},~\eqref{fDK1} and~\eqref{fDFM:1}   that
\begin{equation}  \label{fPSw0} 
 e ^{t H } | \DIbd  |   \le  \| \DIbd  \| _{ \xLn{\infty }  } \CSTtu    h(t )   (\chi  _{ 1 } + e^{-\frac { a _{ 1 }^2} {2at}} ) \le \frac {\CSTK } {2}h(t )   (\chi  _{ 1 } + t^{\frac {m} {2}} )\le \frac {\CSTK} {2}\Theta.
  \end{equation} 
for all $0< t< T$.  If $w\in \Ens$ then by~\eqref{fpPSS4:1} with $z=0$,  
\begin{equation} \label{fPSDZ10} 
|\Md w( s ) |\le   \CSTtd  \CSTK \Theta ( s ) g( s ) =  \CSTtd  \CSTK g( s  ) h( s  ) \Bigl(  s^{\frac {m} {2}} + \sum_{ j=1 }^m s^{\frac {j-1} {2}} \chi _j   \Bigr). 
\end{equation} 
Let now $1\le j\le m$. It follows from~\eqref{fPSRM:b2}, \eqref{fPSchi:b1} and~\eqref{fDFM:4}  that if $|y|>a_j$ and $s<1$, then
\begin{equation} \label{fPSGj} 
  g(s, y) h(s, y)  \le \CSTu  \bigl( 1 +  | y |^{-\frac {N+2} {2} } \bigr)  \le \CSTu  \bigl( 1 +  a_j^{-  \frac {N+2} {2} } \bigr)
  \le \CSTu  \bigl( 1 +  a_m^{ - \frac {N+2} {2} } \bigr) \le \frac {1} {8 T^{\frac {1} {2}} \CSTtu   \CSTtd  }.
\end{equation} 
We deduce from~\eqref{fPSGj} and~\eqref{fPSD} that 
\begin{equation*} 
[e^{ (t-s)  H} g(s) h(s)\chi_j ](x)  \le  \frac {1} {8 T^{\frac {1} {2}}   \CSTtd  } h( t -s ,x) 
 \Bigl(  e^{-\frac { a _{ j+1 }^2} {2 \cstpa (t-s) }} + \chi  _{ j+1 }  \Bigr) 
\le  \frac {1} {8 T^{\frac {1} {2}}   \CSTtd  } h( t   , x)  \Bigl(  e^{-\frac { a _{ j+1 }^2} {2 \cstpa t}} +  \chi  _{ j+1 } \Bigr)
\end{equation*}
and so 
\begin{equation} \label{fPSDZ18}
\int_0^t   e^{(t-s)H}  g(s)h(s)s^{\frac {j-1} {2}}\chi_j  \xdif s \le  \frac {1} {8 T^{\frac {1} {2}}   \CSTtd  } h(t,x) 
 \Bigl(  e^{-\frac { a _{ j+1 }^2} {2 \cstpa t}} + \chi  _{ j+1 }  \Bigr)  t^{\frac {j+1} {2}}. 
\end{equation}   
Since  $t<\frac {1} {4}$, we have $ \sum_{j=1}^\infty  t^{\frac {j+1} {2}}\le \frac {1} {2}$. Thus, \eqref{fPSDZ18} yields 
\begin{equation} \label{fPSDZ19}
\int_0^t e^{(t-s)H}  g(s)h(s)\sum_{j=1}^m s^{\frac {j-1} {2}}\chi_j  \xdif s \le  \frac {1} {8 T^{\frac {1} {2}}   \CSTtd  }  h(t,x)   \Bigl( e^{-\frac { a _{ m+1 }^2} {2 \cstpa t}}+ \sum_{j=1}^m t^{\frac {j+1} {2}}  \chi  _{ j+1 }  \Bigr).
\end{equation}  
Using~\eqref{fPStheta},\eqref{fPSRM} with $\varepsilon =0$, and \eqref{fPSDZ19} we obtain
\begin{equation}  \label{fESM2} 
\begin{split}
\int_0^t e^{(t-s)H} \Theta (s) g(s)  \xdif s 
& \le   h(t)\Biggl ( \CSTR{m}  t^{\frac {m} {2}}+  \frac {1} {8 T^{\frac {1} {2}}   \CSTtd  } \Bigl( e^{-\frac { a _{ m+1 }^2} {2 \cstpa t}}+ \sum_{j=1}^m t^{\frac {j+1} {2}}  \chi  _{ j+1 }  \Bigr) \Biggr )  \\
& =   h(t)\Biggl ( \CSTR{m}  t^{\frac {m} {2}}+  \frac {1} {8 T^{\frac {1} {2}}   \CSTtd  } \Bigl( e^{-\frac { a _{ m+1 }^2} {2 \cstpa t}}+  \sum_{j=2}^{m +1} t^{\frac {j } {2}}  \chi  _{ j }  \Bigr) \Biggr ) .
\end{split}
\end{equation} 
In particular, since $m\ge 2$, we see that
\begin{equation} \label{fESM1} 
\int_0^t e^{(t-s)H} \Theta (s) g(s)  \xdif s \le C t h(t) .
\end{equation} 
Since $\chi _1\ge 0$, $\chi  _{ m+1 }\le 1$, and $e^{-\frac {a_{m+ 1} ^{2}}{2 \cstpa t}} \le  t^{\frac {m + 1 } {2} }$ by~\eqref{fDFM:1}, we deduce from~\eqref{fESM2}  that
\begin{equation*}
\begin{split}
\int_0^t e^{(t-s)H} \Theta (s) g(s)  \xdif s & \le   h(t)\Biggl ( \CSTR{m}  t^{\frac {m} {2}}+  \frac {1} {8 T^{\frac {1} {2}}   \CSTtd  } \Bigl(  t^{\frac {m+1 } {2}} + t^{\frac {1} {2}} \sum_{j=1}^{m } t^{\frac {j-1} {2}}  \chi  _{ j }  \Bigr) \Biggr )  \\
& \le   h(t)\Biggl ( \CSTR{m}  t^{\frac {m} {2}}+  t^{\frac {1} {2}}  \frac {1} {8 T^{\frac {1} {2}}   \CSTtd  }\Bigl(  t^{\frac {m } {2}} +  \sum_{j=1}^{m } t^{\frac {j-1} {2}}  \chi  _{ j }  \Bigr) \Biggr ) .
\end{split}
\end{equation*} 
Applying~\eqref{fDRM} and~\eqref{fDFM:4} we obtain
\begin{equation} \label{fPSDZ20b}
\int_0^t e^{(t-s)H} \Theta (s) g(s)  \xdif s  \le \frac {1} {2  \CSTtd  }   h(t)\Bigl (  t^{\frac {m} {2}}+  \sum_{j=1}^{m } t^{\frac {j-1} {2}}  \chi  _{ j }  \Bigr)  =  \frac {1} {2  \CSTtd  }  \Theta (t) .
\end{equation} 
Using the first inequality in~\eqref{fPSDZ10}, we deduce that
\begin{equation} \label{fPSDZ20}
\begin{split}
\Bigl |\int_0^t &e^{(t-s)H} \Md w(s)  \xdif s\Bigr |\le \frac {\CSTK} {2}\Theta .
\end{split}
\end{equation} 
With the notation~\eqref{fNVDP1},  it follows from~\eqref{fPSDZ20} and~\eqref{fPSw0} that $\Phi (\Ens)\subset \Ens$. Moreover,  if $w,z\in \Ens$, we deduce from~\eqref{fpPSS4:1} and~\eqref{fPSDZ14:b1}  that 
\begin{equation*} 
|\Md w(s)-\Md z(s)|\le   \CSTtd  |w (s) -z (s) | g(s)\le   \CSTtd  d(w,z)\Theta (s) g (s) .
\end{equation*} 
Using again~\eqref{fPSDZ20b}  we obtain  
\begin{equation*}
\begin{split}
|\Phi (w)(t)-\Phi(z)(t)| \le  \int_0^t &e^{(t-s)H} |\Md w(s)-\Md z(s)|  \xdif s\le \frac {1} {2}\Theta (t) d(w,z).
\end{split}
\end{equation*} 
Therefore, $\dist ( \Phi (w), \Phi (z)) \le \frac {1} {2} d(w,z)$, so that   $\Phi $ has a unique fixed point $w \in  \Ens$.  
We deduce from Lemma~\ref{eLemSol1} that
\begin{equation*} 
\partial _t w - \Delta w - \beta (\alpha +1)  |x|^{-2} w = \Md w 
\end{equation*} 
in ${\mathcal D}' ((0, T) \times  \R^N )$. 
Therefore,
\begin{equation} \label{fEquw} 
\partial _t w - \Delta w   =  | U + w|^\alpha (U +w) -  |U|^\alpha U 
\end{equation} 
in ${\mathcal D}' ((0, T) \times \R^N )$, by~\eqref{fPT3}.  
Since the right-hand side of the above equation is in $\xLn{\infty } _\Loc ( (0,T) \times (\R^N \setminus \{0\} )$ by $ w\in \Ens $ and~\eqref{fPShrho}, it follows from Lemma~\ref{eLemLHISP7} that  $w \in \mathrm{C} ((0,T) \times ( \R^N \setminus \{0\})$. 

We now set
\begin{equation} \label{fDfnuu} 
u = U + w ,
\end{equation} 
so that $u\in \xLn{{\alpha +1}} _\Loc ((0,T) \times \R^N ) \cap \mathrm{C} (( 0,T) \times ( \R^N \setminus \{0\}  ) )$ is a solution of~\eqref{fPS4:b8} in ${\mathcal D}' ((0, T) \times   \R^N  )$.
Estimate~\eqref{fPS4:b7} follows from~\eqref{fDfnuu} and $w\in \Ens$.
Since $w\in \Ens$, we have $ |\Md w | \le C  \Theta g $ by~\eqref{fpPSS4:1} with $z=0$, so it follows from~\eqref{fESM1} that
\begin{equation} \label{fEstwmwz} 
 | w(t)- e^{t H } \DIbd | \le C t h(t) .
\end{equation} 
Finally, since $u\in  \mathrm{C} (( 0,T) \times ( \R^N \setminus \{0\}  ) )$, we see that $u(t,x)$ is defined for all $0<t<T$ and $x\not = 0$, and using~\eqref{fPS4:b7}, we deduce that $u(t) \in \xLn{1} _\Loc (\R^N ) $ for all $0<t<T$.
Moreover, given $R>0$,
\begin{equation*} 
\begin{split} 
 \|u(t) - \DI \| _{ \xLn{1}  (\{   |x| < R \} ) }  \le &  \|U(t) - U(0)   \| _{ \xLn{1} (\{   |x| < R \} ) } + \|  w(t)  - e^{tH } \DIbd  \| _{ \xLn{1} (\{   |x| < R \} ) } 
\\ & +    \|  e^{tH } \DIbd - \DIbd \| _{ \xLn{1} (\{   |x| < R \} ) }  \goto  _{ t\to 0 } 0,
\end{split} 
\end{equation*} 
by~\eqref{fDfnUdz}, \eqref{fEstwmwz},  and  Lemma~\ref{eLemLHISP4}. This proves~\eqref{fPS4:b8} and completes the proof of Theorem~\ref{ePSSZ1}. 

\section{Singular, radially symmetric, stationary solutions} \label{sStatSing} 
Let 
\begin{equation} \label{fStat1} 
N\ge 3 \quad  \text{and}\quad  \frac {2} {N-2} < \alpha  < \frac {4} {N-2} ,
\end{equation} 
and let $\beta >0$ be defined by~\eqref{bta}. 
We study radially symmetric, possibly sign-changing stationary solutions of~\eqref{NLHE}. 
Such solutions satisfy the ODE
\begin{equation} \label{fStat2} 
u'' + \frac {N-1} {r} u' +  |u|^\alpha  u=0.
\end{equation} 
It follows easily from energy arguments that any solution of~\eqref{fStat2} on some interval $(a,b)$ with $0\le a< b\le \infty $  can be extended to a solution on $(0, \infty )$.  

All the positive solutions of~\eqref{fStat2} are known, see~\cite[Proposition~3.1]{SerrinZ}. They consist of the solution $  \beta ^{\frac {1} {\alpha }} |x|^{- \frac {2} {\alpha }}$ and a one-parameter family  $(u _\lambda )  _{ \lambda >0 } \subset \xCn{2} ( 0,\infty )$ 
of singular solutions satisfying $ r^{\frac {2} {\alpha }} u_\lambda (r) \to \beta ^{\frac {1} {\alpha }}  $ as $r\to 0$ and $ r ^{N-2} u_\lambda (x) \to \lambda $ as $ r \to \infty $.
The negative solutions of~\eqref{fStat2} are therefore $ - \beta ^{\frac {1} {\alpha }} |x|^{- \frac {2} {\alpha }}$ and   $(- u _\lambda )  _{ \lambda >0 }$.
All other solutions of~\eqref{fStat2} are therefore sign-changing, and they consist of the one-parameter family of regular solutions, i.e. the solutions of~\eqref{fStat2}  with the initial conditions $u(0)= c\in \R$, $u'(0)=0$ and the (two-parameter family) sign-changing, singular solutions. More precisely, we have the following result.

\begin{prpstn} \label{eStat1} 
Assume~\eqref{fStat1} and let $\beta >0$ be defined by~\eqref{bta}. 
If $u\in \xCn{2} (0,\infty )$, $u\not \equiv 0$, is a solution of~\eqref{fStat2}, then the following properties hold.
\begin{enumerate}[{\rm (i)}] 

\item \label{eStat1:z1} 
$ |u(r) |\le  C  r ^{- (N-2 - \frac {2} {\alpha })}$ for $r> 1$.
In addition, either $u(r)$ has a finite limit as $r\to 0$, or else $r^{\frac {2} {\alpha }} u(r) \to \pm  \beta ^{\frac {1} {\alpha }}  $ as $r\to 0$. 
In the first case, $u ' (r)\to 0$ as $r\to 0$, and in the second case $r^{\frac {2} {\alpha } +1} u'(r) \to \mp \frac {2} {\alpha }  \beta ^{\frac {1} {\alpha }} $. 

\item \label{eStat1:2} 
$u$ satisfies one of the following properties.

\begin{enumerate}[{\rm (a)}] 
\item \label{eStat1:2:1} 
$u$ is a regular solution of~\eqref{fStat2}, i.e. $u\in \xCn{2} ([0,\infty ))$ and $u'(0)=0$. In this case, $u$ oscillates indefinitely as $r\to \infty $. 

\item \label{eStat1:2:2}
$u$ is a constant-sign, singular solution of~\eqref{fStat2},  i.e. either $u = \pm \beta ^{\frac {1} {\alpha }} r^{- \frac {2} {\alpha }}$ or else $u=  \pm u_\lambda $  for some $\lambda >0$ where $u_\lambda $ is as defined above. 

\item \label{eStat1:2:3}
$u$ is a singular, sign-changing solution. In this case, $r^{\frac {2} {\alpha }} u(r) \to \pm \beta ^{\frac {1} {\alpha }} $ as $r\to 0$, and $u$ oscillates indefinitely as $r \to \infty $.
\end{enumerate} 

\item \label{eStat1:3} 
$u$ satisfies~\eqref{eStat1:2:3} if and only if there exists $r_0>0$ such that 
\begin{equation} \label{feStat1:3} 
\frac {r_0^2} {2}  |u'(r_0)|^2 + \frac {2} {\alpha } r_0 u(r_0) u'(r_0) + \frac {1} {\alpha +2} r_0^2  |u(r_0)|^{\alpha +2} + \frac {4- (N-2)\alpha } {\alpha ^2}  |u(r_0)|^2 =0 .
\end{equation} 

\item \label{eStat1:1} 
Let $U\in \xCn{2} ( \R^N \setminus \{ 0 \} )$ be defined by $U(x)= u( |x| ) $ for $x\not = 0$. It follows that $U\in \xLn{{\alpha +1}}_\Loc (\R^N )$ and that $U$ is a  solution of 
\begin{equation} \label{feStat1:1} 
- \Delta U=  |U|^\alpha U
\end{equation} 
 in ${\mathcal D}' (\R^N )$.

\end{enumerate} 

\end{prpstn} 

\begin{proof} 
Given a solution $u\in C^2(0,\infty )$ of~\eqref{fStat2}, we set (see~\cite[equation~(6)]{MazzeoP})
\begin{equation} \label{fStat3} 
u (r) = r^{- \frac {2} {\alpha }} v(s), \quad s= - \log r, 
\end{equation} 
so that $v\in \xCn{2} (\R) $ satisfies
\begin{equation}  \label{fStat4}
v '' + \gamma v' - \beta v +  |v|^\alpha v=0
\end{equation} 
for all $s\in \R$, where $\beta $ is given by~\eqref{bta} and
\begin{equation}  \label{fStat5}
 \gamma = \frac {4} {\alpha }- N + 2 >0. 
\end{equation} 
(Note that we use~\eqref{fStat1} to obtain $\beta , \gamma >0$.)
Equation~\eqref{fStat4} has the stationary solutions $0$ and $\pm \beta ^{\frac {1} {\alpha }}$, corresponding to the solutions $u(r) \equiv 0$ and $u(r) \equiv  \pm \beta ^{\frac {1} {\alpha }} r^{- \frac {2} {\alpha }}$ of~\eqref{fStat2}. 
We now suppose that  $v$ is not a stationary solution of~\eqref{fStat4}.
Setting 
\begin{equation}  \label{fStat6} 
F(v, v')= \frac {1} {2}  |v'|^2 + \frac {1} {\alpha +2}  |v|^{\alpha +2} - \frac {\beta  } {2}  |v|^2,
\end{equation} 
we see that 
\begin{equation} \label{fStat7} 
f(s) = : F(v(s), v'(s)) 
\end{equation} 
satisfies
\begin{equation} \label{fStat8} 
f'(s) + \gamma  | v'(s)|^2 = 0.
\end{equation} 
We first note that, since $v$ is not a stationary solution, $\int  _{ s_0 - 1 }^{s_0 }  |v'|^2 >0$ for all $s<s_0$, and
it follows from~\eqref{fStat8}, that 
\begin{equation} \label{fStat8:b1} 
f(s_1) > f(s_2)  \quad  \text{for all}\quad s_1<s_2. 
\end{equation} 
Next,
\begin{equation} \label{fStat9} 
F(v, v') \ge F( \beta ^{\frac {1} {\alpha }}, 0 )= - \frac {\alpha } {2( \alpha +2) } \beta ^{\frac {\alpha +2} {\alpha }}= : F_\star .
\end{equation} 
It follows that $f$ is bounded from below, so that $f (s)$ decreases to a limit $f_\infty $ as $s\to \infty $.
We deduce easily that
\begin{equation} \label{fSupfin1} 
\sup  _{ s\ge 0 } (  |v(s)| +  | v' (s) |) <\infty ,
\end{equation} 
 and
\begin{equation} \label{fStat10}
\gamma \int _s ^\infty | v'(\tau )|^2 \xdif \tau = f(s) - f_\infty  <\infty . 
\end{equation} 
Using~\eqref{fSupfin1}-\eqref{fStat10}, it follows by standard arguments that 
\begin{equation} \label{fStat10:b1}
v'(s) \goto _{ s\to \infty  }0, \quad v(s) \goto _{ s\to \infty  }\ell, \quad  \text{with}\quad \ell=0 \quad  \text{or}\quad \ell=    \pm \beta ^{\frac {1} {\alpha }} .
\end{equation} 
We now proceed in seven steps.

\Step1 The case $\ell= 0$ in~\eqref{fStat10:b1}. This corresponds to the case where $u(r) $ has a finite limit as $r\to 0$, and then $u'(r) \to 0$ as $r\to 0$. (See~\cite[formula~(2.4)]{SerrinZ}.)
Note that in this case, $f(s) \to 0$ as $s\to \infty $, so that by~\eqref{fStat8:b1}
\begin{equation}  \label{fStat13:b6} 
f (s) >0,\quad s\in \R. 
\end{equation} 

\Step2 The case $ \ell = \pm \beta ^{\frac {1} {\alpha }}$  in~\eqref{fStat10:b1}. 
This means that 
\begin{equation}  \label{fStat13:b5} 
 r^{\frac {2} {\alpha }} u(r) \goto  _{ r\to 0 } \pm \beta ^{\frac {1} {\alpha }}. 
\end{equation} 
In addition, 
\begin{equation}  \label{fStat13:b4} 
r^{\frac {2} {\alpha } +1} u'(r)= - \frac {2} {\alpha } v(s) - v'(s) \goto  _{ r\to 0 } \mp \frac {2} {\alpha }  \beta ^{\frac {1} {\alpha }} .
\end{equation} 
Moreover, 
\begin{equation}  \label{fStat13:b3} 
f(s) \goto  _{ s \to \infty } F_\star <0 . 
\end{equation} 

\Step3 We prove Property~\eqref{eStat1:z1}. 
We first prove that $ |u(r) |\le  C  r ^{- (N-2 - \frac {2} {\alpha })}$ for $r> 1$.
In terms of $v$, this means
\begin{equation} \label{fStat19}
 |v (s) | \le C e ^{- \gamma s} , \quad s\le 0. 
\end{equation} 
Integrating~\eqref{fStat8}  on $(s,0)$ we obtain
\begin{equation} \label{fStat17}
 |v '|^2  = \beta  | v |^2 -  \frac {2} {\alpha +2}  | v |^{\alpha +2} + 2 F( v(0), v '(0) ) + 2\gamma  \int _s^0  | v ' (\sigma )| ^2 \xdif \sigma 
  \le C +  2\gamma  \int _s^0  | v ' (\sigma )| ^2  \xdif \sigma,
\end{equation} 
for some constant $C$ so that, by Gronwall's inequality, $ | v ' (s) | \le C e ^{- \gamma s}$ for $s<0$. Estimate~\eqref{fStat19} follows by integration.
The part of Property~\eqref{eStat1:z1} concerning the behavior of $u(r)$ as $r\to 0$ follows from Steps~1 and~2.

\Step4 We prove that if $f(s_0) \ge 0$ for some $s_0\in \R$, then $v$ oscillates indefinitely as $s\to -\infty $. 
Suppose by contradiction that (for instance) $v(s) >0$ for $s< s_1$ with $s_1 < s_0$. This means that $u (r) >0$ for $r$ large. It follows from~\cite[Lemma~2.0]{SerrinZ} that $u (r)>0$ for all $r>0$. 
Applying now~\cite[Proposition~3.1]{SerrinZ}, we deduce that $u$ is either the solution $  \beta ^{\frac {1} {\alpha }} |x|^{- \frac {2} {\alpha }}$ or a solution  $u_\lambda $ for some $\lambda >0$. 
The first case corresponds to $v(s) \equiv \beta ^{\frac {1} {\alpha }} $, so that $f(s) \equiv  F_\star <0$, which is absurd.
In the second case, $u (r) \le c r^{- (N-2)}$, so that $v(s) \to 0$ as $s\to - \infty $. Therefore, $v'(s_n) \to 0$ for some sequence $s_n \to -\infty $, so that $f(s_n) \to 0$ as $n\to \infty $. 
On the other hand, it follows from~\eqref{fStat8:b1} that for $n$ large, $f(s_n) > f (s_0 -1) >0$.
This is again absurd.

\Step5 We prove Property~\eqref{eStat1:2}. 
We consider three cases.

-- If $f(s) >0$ for all $s\in \R$, then $v$ oscillates indefinitely as $s\to -\infty $ by Step~4; and so, $u$ oscillates indefinitely as $r\to \infty $. Moreover, it follows from~\eqref{fStat13:b3} that $ \ell = 0$  in~\eqref{fStat10:b1}.
By Step~1, this implies that  $u$ satisfies~\eqref{eStat1:2:1}. 

-- If $f(s) <0$ for all $s\in \R$, then $ \ell = \pm \beta ^{\frac {1} {\alpha }}$  in~\eqref{fStat10:b1}. 
Indeed, the case $\ell=0$ is ruled out by~\eqref{fStat13:b6}. Moreover, $v$ cannot vanish, for if $v(s)=0$, then $f(s) >0$. 
Therefore, $u$ is a constant-sign singular solution of~\eqref{fStat2}, and it follows from~\cite[Proposition~3.1]{SerrinZ} that 
$u$ satisfies~\eqref{eStat1:2:2}. 

-- If $f(s_0 ) =0$ for some $s_0\in \R$, then $v$ oscillates indefinitely as $s\to -\infty $ by Step~4; and so, $u$ oscillates indefinitely as $r\to \infty $. In addition, it follows from~\eqref{fStat8:b1} that $f(s) <0$ for $s >s_0$. Therefore, $ \ell = \pm \beta ^{\frac {1} {\alpha }}$  in~\eqref{fStat10:b1}. Indeed, the case $\ell=0$ is ruled out by~\eqref{fStat13:b6}. Therefore, $u$ is a singular solution of~\eqref{fStat2} that oscillates indefinitely as $r\to \infty $. This means that $u$ satisfies~\eqref{eStat1:2:3}. 

\Step6 We prove Property~\eqref{eStat1:3}. 
The discussion in Step~5 shows that $u$ satisfies~\eqref{eStat1:2:3} if and only $f(s_0) =0$ for some $s_0 \in \R$. 
This is equivalent to~\eqref{feStat1:3} with $r_0= e^{- s_0}$. 
Note that $f$ is decreasing by~\eqref{fStat8:b1}, so that $f$ can vanish at most for one value of $s$, and therefore~\eqref{feStat1:3} can be satisfied at most for one value of $r_0$. 

\Step7 We prove Property~\eqref{eStat1:1}. 
This is well known if $u$ is a regular solution. In the general case, $ |U(x)|\le C (1 +  |x|^{-\frac {2} {\alpha } })$ by Property~\eqref{eStat1:z1}; and so $U\in \xLn{{\alpha +1}} _\Loc (\R^N )$ by~\eqref{fStat1}. Let now $\varphi \in \mathrm{C} ^\infty _\Comp (  \R^N)$. 
Since $U\in \xLn{{\alpha +1}} _\Loc (\R^N )$, we see that
\begin{equation} \label{fStat18} 
\Bigl|  \int_{ \{  |x|<\varepsilon  \}} ( U \Delta \varphi  + | U |^\alpha U \varphi  ) \Bigr|  \le C   \int_{ \{  |x|<\varepsilon  \}} ( |U | +  | U |^{ \alpha +1})\goto _{ \varepsilon \downarrow 0 } 0 .
\end{equation}  
On the other hand, $U $ is a classical solution of~\eqref{feStat1:1} on $  \R^N  \setminus \{ 0\} $, so that integration by parts yields
\begin{equation*} 
\int_{ \{  |x| >\varepsilon  \}} ( U   \Delta \varphi   + | U |^\alpha U  \varphi ) = - \int_{ \{  |x|=\varepsilon  \}}  \Bigl( U  \frac {\partial \varphi } {\partial r}- \varphi \frac {\partial  U } {\partial r}  \Bigr)    .
\end{equation*} 
Since $\varepsilon ^{N-1} (  | u(\varepsilon )| +  | u' (\varepsilon )|  ) \to 0$ as $\varepsilon \to 0$ by Property~\eqref{eStat1:z1},
we deduce that
\begin{equation} \label{fStat19:b1} 
\int_{ \{  |x| >\varepsilon  \}} ( U   \Delta \varphi   + | U |^\alpha U  \varphi ) \goto _{ \varepsilon \to 0  } 0  .
\end{equation} 
Estimates~\eqref{fStat18} and~\eqref{fStat19:b1} imply 
\begin{equation*} 
\int_{\R^N}  (U  \Delta \varphi   + | U |^\alpha U \varphi )=0 .
\end{equation*}  
Since $\varphi \in \xCn{\infty } _\Comp (  \R^N)$ is arbitrary, we see that $U$ is a  solution of~\eqref{feStat1:1}  in ${\mathcal D}' (\R^N )$.
\end{proof} 

\begin{crllr} \label{eRemSum} 
Assuming~\eqref{fStat1}, there exists a two-parameter family of radially symmetric, sign-changing, stationary solutions $U\in \xCn{2} ( \R^N \setminus \{ 0 \} ) \cap \xLn{{\alpha +1}}_\Loc (\R^N )$ of~\eqref{NLHE} in the sense of distributions, which are singular at $x=0$.  
These solutions satisfy $ |U( x) |\le  C   | x | ^{- (N-2 - \frac {2} {\alpha })}$ for $ |x|> 1$ and oscillate indefinitely as $ |x| \to \infty $. 
\end{crllr} 

\begin{proof} 
Applying Proposition~\ref{eStat1}, we need only show that there is a two-parameter family of solutions of~\eqref{fStat2} that satisfy~\eqref{feStat1:3} for some $r_0>0$ (depending on the solution).  
To see this, consider any  $r_0>0$ and $0 < a < (\frac {\alpha +2} {2 })^{\frac {1} {\alpha }}$. Let $b> \frac {2a} {\alpha }$ be defined by
\begin{equation*} 
\frac {1} {2}  \Bigl( b - \frac {2a} {\alpha } \Bigr)^2 = a^2 \frac {(N-2) \alpha -2} {\alpha ^2 (\alpha +2) } [ \alpha +2 - 2 a^\alpha  ] .
\end{equation*} 
If $u$ is the solution of~\eqref{fStat2} defined by $u( r_0) = a \beta ^{\frac {1} {\alpha }} r_0^{- \frac {2} {\alpha }}$ and $u ' ( r_0) = -b  \beta ^{\frac {1} {\alpha }} r_0^{- \frac {2} {\alpha } -1}$, then~\eqref{feStat1:3} is satisfied. 
It remains to prove that two different choices of $(r_0, a )$ yield two different solutions of~\eqref{fStat2}. 
Suppose two choices $r_0^1, r_0^2 >0$ and $a^1, a^2 \in (0,  (\frac {\alpha +2} {2 })^{\frac {1} {\alpha }})$ produce the same solution $u$ of~\eqref{fStat2}. 
With the notation used in the proof of Proposition~\ref{eStat1}, we have $f( s_0^1 ) = f( s_0^2 )$, where $s_0^j = - \log r_0^ j$, $j=1,2$. (See Step~6 of the proof of Proposition~\ref{eStat1}.)
Note that $ u(r)  \not \equiv 0$ since $a^1 \not = 0$, and that $u (r)  \not \equiv \beta ^{\frac {1} {\alpha }}  r^{- \frac {2} {\alpha }}$ since $u$ is sign-changing. Therefore, it follows from~\eqref{fStat8:b1} that $s_0^1 = s_0^2$, so that $r_0^1= r_0^2$. Since $u( r_0^j) = a^j \beta ^{\frac {1} {\alpha }} (r_0^j )^{- \frac {2} {\alpha }}$, $j=1,2$, we conclude that also $a^1= a^2$. This completes the proof.
\end{proof} 

\section{Proof of Theorems~\ref{ePSSZpr1} and~\ref{ePSSZpr2}} \label{sSingProf} 

In this section, we give the proof of Theorems~\ref{ePSSZpr1} and~\ref{ePSSZpr2}, which are consequences of Theorem~\ref{ePSSZ1}.
We begin with the following proposition, which provides refined estimates for the behavior at the origin of both the radially symmetric singular stationary solutions of~\eqref{NLHE} and of the profiles of singular, radially symmetric, self-similar solutions of~\eqref{NLHE}.

\begin{prpstn} \label{eLemCV1} 
Assume~\eqref{fCondAlpha}. Let $a\in \R$ and $u\in \xCn{2}( 0,\infty )$ a solution of
\begin{equation} \label{feLemCV1:1} 
u'' + \Big(\frac{N-1}{r} +  \frac{a r}{2}\Big) u '  + \frac{a}{\alpha} u  +  | u |^{\alpha} u = 0.
\end{equation} 
If
\begin{equation} \label{feLemCV1:2} 
r^{\frac {2} {\alpha }} u(r) \goto _{ r\to 0 } \beta ^{\frac {1} {\alpha }},
\end{equation} 
where $\beta $ is defined by~\eqref{bta}, then there exists a constant $C$ such that
\begin{equation} \label{feLemCV1:3}
| r^{\frac {2} {\alpha }} u (r) - \beta ^{\frac {1} {\alpha }} |  \le C r^\rho ,
\end{equation} 
for all $0<r\le 1$, where $\rho >0$ is defined by~\eqref{fPSrho}.
\end{prpstn} 

\begin{proof} 
Setting
\begin{equation*} 
u (r) = r^{- \frac {2} {\alpha }} v(s), \quad s= - \log r, 
\end{equation*} 
the equation for $v$ is
\begin{equation*}
v '' + ( \gamma - a e^{-2s} ) v' - \beta v +  |v|^\alpha v=0
\end{equation*} 
where 
\begin{equation*}
\gamma = \frac {4} {\alpha }- N + 2 >0. 
\end{equation*} 
Let 
\begin{equation*} 
v= \beta ^{\frac {1} {\alpha }} + z ,
\end{equation*} 
and set 
\begin{equation*} 
g(s)=  |s|^\alpha s,\quad \varphi (s)= g( \beta ^{\frac {1} {\alpha }} + s) - g (\beta ^{\frac {1} {\alpha }} ) - g' (\beta ^{\frac {1} {\alpha }} ) s.
\end{equation*} 
It follows that 
\begin{equation*} 
z'' + \gamma z' + \alpha \beta z= h, 
\end{equation*} 
where
\begin{equation} \label{fDefnH1} 
h = a e^{-2s} z'  - \varphi (z) . 
\end{equation} 
Note that the fundamental solutions of the linear equation $z'' + \gamma z' + \alpha \beta z=  0$ are 
\begin{equation*} 
\phi _1 (s) = e^{- 2 \mu _1 s}, \quad \phi _2 (s) = e^{- 2 \mu _2 s}, 
\end{equation*} 
where $\mu _1, \mu _2$ are given by Lemma~\ref{eNUM1}. 
By the variation of the parameter formula, we deduce that 
\begin{equation} \label{feLemCV1:4}
z(s)=   \Bigl( C_1 + \frac {1} {2(\mu _2- \mu _1) } \int  _{ s_0 } ^s e^{2\mu _1 \sigma } h  \xdif \sigma  \Bigr) e^{- 2 \mu _1 s}  +  \Bigl( C_2 -  \frac {1} {2(\mu _2- \mu _1) } \int  _{ s_0 } ^s e^{2\mu _2 \sigma } h \xdif \sigma  \Bigr) e^{- 2 \mu _2 s} ,
\end{equation} 
for some appropriate constants $C_1, C_2$, and
\begin{equation} \label{feLemCV1:5}
z' (s)=  -2\mu _1 \Bigl( C_1 + \frac {1} {2(\mu _2- \mu _1) } \int  _{ s_0 } ^s e^{2\mu _1 \sigma } h  \xdif \sigma  \Bigr) e^{- 2 \mu _1 s}  -2\mu _2  \Bigl( C_2 -  \frac {1} {2(\mu _2- \mu _1) } \int  _{ s_0 } ^s e^{2\mu _2 \sigma } h  \xdif \sigma  \Bigr) e^{- 2 \mu _2 s} .
\end{equation} 
We first prove that there exists $\varepsilon >0$ such that
\begin{equation} \label{feLemCV1:6}
 | z(s) | +  |z'(s)| \le C e^{- \varepsilon s}, \quad s\ge 0.
\end{equation} 
To see this, we fix $\varepsilon >0$ sufficiently small so that
\begin{equation} \label{feLemCV1:7}
\gamma -4\varepsilon >0, \quad \alpha \beta + 4\varepsilon ^2 - 2\gamma \varepsilon >0 ,
\end{equation} 
and we set
\begin{equation} \label{feLemCV1:8}
\eta _1 = \frac {\gamma -4\varepsilon} {2} >0, \quad \eta _2 = \frac { \alpha \beta + 4\varepsilon ^2 - 2\gamma \varepsilon } {2}>0 .
\end{equation} 
Setting 
\begin{equation} \label{feLemCV1:8:b1}
z (s) = e^{- 2 \varepsilon s} \psi (s), 
\end{equation} 
we obtain for $\psi $ the equation
\begin{equation*} 
\psi '' + 2\eta _1 \psi ' + 2 \eta _2 \psi  = 
a e^{-2 s} ( \psi ' - 2\varepsilon \psi ) - e^{ 2 \varepsilon s } \varphi (z)  . 
\end{equation*} 
Multiplying by $\psi '$, we obtain
\begin{equation*} 
\frac {d} {ds}  \Bigl( \frac {1} {2}  |\psi '|^2 + \eta _2 \psi ^2 \Bigr) + (2 \eta _1 - a e^{-2 s}  ) | \psi ' |^2 = - 2 \varepsilon a e^{-2 s}\psi \psi ' - e^{ 2 \varepsilon s } \varphi (z)  \psi ' .
\end{equation*} 
Now
\begin{equation*} 
 - 2 \varepsilon a e^{-2 s}\psi \psi ' - e^{ 2 \varepsilon s } \varphi  (z)  \psi ' \le \varepsilon a e^{-2 s}  \psi ^2 + \varepsilon a e^{-2 s}  | \psi ' |^2 
 + \eta _1  | \psi ' |^2  + \frac {1} {4 \eta _1} ( e^{ 2 \varepsilon s }  \varphi (z))^2. 
\end{equation*} 
We first choose $s_0$ sufficiently large so that 
\begin{equation*} 
a e^{-2 s} \le \eta _2\quad  \text{and} \quad \varepsilon a e^{-2 s} \le \eta _1 - a e^{-2 s}  \quad  \text{for }s\ge s_0,
\end{equation*} 
and we obtain
\begin{equation*} 
\frac {d} {ds}  \Bigl( \frac {1} {2}  |\psi '|^2 + \eta _2 \psi ^2 \Bigr)   \le \varepsilon \eta_2  \psi ^2  + \frac {1} {4 \eta _1} ( e^{ 2 \varepsilon s }  \varphi (z))^2. 
\end{equation*} 
Next we observe that there exists a constant $A$ such that 
\begin{equation} \label{feLemCV1:9}
  |\varphi (z)| \le A z^2, \quad  |z|\le 1.
\end{equation} 
Note that by~\eqref{feLemCV1:2} $ |z (s)| \to 0$ as $s\to \infty $, so that by possibly choosing $s_0$  larger we have $ |z(s)|\le 1$ for $s\ge  s_0 $. 
Therefore
\begin{equation*} 
 | e^{ 2 \varepsilon s } \varphi ( z ) | \le A e^{ 2 \varepsilon s } z^2 = A  |z| \,  |\psi | .
\end{equation*} 
Thus we see that
\begin{equation*} 
\frac {d} {ds}  \Bigl( \frac {1} {2}  |\psi '|^2 + \eta _2 \psi ^2 \Bigr)   \le \varepsilon \eta_2  \psi ^2  + \frac {A^2 z^2 } {4 \eta _1} \psi ^2.  
\end{equation*} 
By choosing $s_0$ possibly larger, we deduce from~\eqref{feLemCV1:2} that
\begin{equation*} 
\frac {A^2 z^2 } {4 \eta _1} \le \varepsilon \eta_2 \quad s\ge s_0,
\end{equation*} 
and we conclude that
\begin{equation*} 
\frac {d} {ds}  \Bigl( \frac {1} {2}  |\psi '|^2 + \eta _2 \psi ^2 \Bigr)   \le 2 \varepsilon \eta_2  \psi ^2 \le 2\varepsilon \Bigl( \frac {1} {2}  |\psi '|^2 + \eta _2 \psi ^2 \Bigr)  ,
\end{equation*} 
for $s \ge s_0$. Applying Gronwall's inequality, we deduce that
\begin{equation*} 
 |\psi | +  |\psi '| \le C e^{\varepsilon s},
\end{equation*} 
Using~\eqref{feLemCV1:8:b1}, the claim~\eqref{feLemCV1:6} follows.

We complete the proof by using formulas~\eqref{feLemCV1:4}-\eqref{feLemCV1:5} and a bootstrap argument.
We note that $\rho = 2\mu _1$ by Lemma~\ref{eNUM1}, so that~\eqref{feLemCV1:3} is proved if 
\begin{equation} \label{feLemCV1:3:b1} 
 | z(s) | +  |z'(s)| \le C e^{- 2 \mu _1 s}, \quad s\ge 0.
\end{equation} 
Suppose that 
\begin{equation} \label{feLemCV1:10}
 | z(s) | +  |z'(s)| \le C e^{- \nu s}, \quad s\ge 0,
\end{equation} 
for some $\nu >0$. 
It follows from~\eqref{fDefnH1}, \eqref{feLemCV1:10}  and~\eqref{feLemCV1:9} that
\begin{equation*} 
 |h | \le C e^{-  \widetilde{ \nu } },
\end{equation*} 
where
\begin{equation}  \label{feLemCV1:11}
 \widetilde{ \nu }= \min \{ \nu + 2 , 2\nu \} = \nu + \min \{  2 ,  \nu \}  > \nu .
\end{equation} 
It is not difficult to deduce from~\eqref{feLemCV1:4}-\eqref{feLemCV1:5} that
\begin{equation}  \label{feLemCV1:12}
 | z(s) | +  |z'(s)| \le 
 \begin{cases} 
C e^{-  \widetilde{ \nu } } &  \widetilde{ \nu } < 2\mu _1 ,\\
C (1+s) e^{- 2 \mu _1} & \widetilde{ \nu }= 2\mu _1 ,\\
C  e^{- 2 \mu _1}  & \widetilde{ \nu } > 2\mu _1 .
 \end{cases} 
\end{equation} 
We now conclude as follows. 
It follows from~\eqref{feLemCV1:6} that~\eqref{feLemCV1:10} holds for $ \widetilde{ \nu} =\varepsilon >0$.
If $\varepsilon >2\mu _1$, then the estimate~\eqref{feLemCV1:3:b1} follows from~\eqref{feLemCV1:12}.  
If $\varepsilon \le 2\mu _1$, then it follows from~\eqref{feLemCV1:12} that~\eqref{feLemCV1:10} holds with
\begin{equation*} 
\nu= \varepsilon + \frac {1} {2} \min \{  2 ,  \varepsilon  \}. 
\end{equation*} 
We now can iterate the above argument. In at most $\ell $ steps, where
\begin{equation*} 
\frac {\ell } {2} \min \{  2 ,  \varepsilon  \} > 2 \mu _1,
\end{equation*} 
we obtain the estimate~\eqref{feLemCV1:3:b1}. 
\end{proof} 

We are now in a position to prove Theorem~\ref{ePSSZpr1}.

\begin{proof} [Proof of Theorem~$\ref{ePSSZpr1}$]
Let $U$ be a radially symmetric, stationary solution of~\eqref{NLHE}  that is singular at $x=0$.
It follows that $U(x) = u(  |x|) $ where $u\in C^2(0,\infty )$, $u\not \equiv 0$, is a solution of~\eqref{fStat2} which is singular at $r=0$. We deduce from Proposition~\ref{eStat1}~\eqref{eStat1:2} that $r^{\frac {2} {\alpha }} u(r) \to \pm \beta ^{\frac {1} {\alpha }} $ as $r\to 0$. Therefore, we may apply Proposition~\ref{eLemCV1} with $a=0$, and we obtain that $u$ satisfies~\eqref{feLemCV1:3}.
Since $u$ is bounded as $r\to \infty $ by Proposition~\ref{eStat1}~\eqref{eStat1:z1}, it follows that $U$ satisfies~\eqref{ePSSZ2}, and also~\eqref{fPShrho} for all $0<t<1$ and $x\not = 0$, while~\eqref{fDfnUdz} is trivial with $U_0 = U$.  
Since $\partial _t U - \Delta U -  |U|^\alpha U=0 $ in ${\mathcal D}' ((0,T) \times \R^N )$ by Proposition~\ref{eStat1}~\eqref{eStat1:1}, the result now follows by applying Theorem~\ref{ePSSZ1} with $S=1$.
\end{proof} 

For the proof of Theorem~\ref{ePSSZpr2}, we will use the following proposition.

\begin{prpstn} \label{eLemCV2} 
Assume~\eqref{fStat1} and let $f\in \xCn{2} (0,\infty  ) $ be a solution of the equation~\eqref{fpr2}. 
It follows that 
\begin{equation} \label{feLemCV2:1} 
\ell = \lim  _{ r\to 0  } r^{\frac {2} {\alpha }} f(r) 
\end{equation} 
exists and either $\ell =0$ or else $\ell = \pm \beta ^{\frac {1} {\alpha }}$, where $\beta $ is defined by~\eqref{bta}; and
\begin{equation} \label{feLemCV2:1:b1} 
 \lim  _{ r\to 0  } r^{ 1 + \frac {2} {\alpha }} f'(r) = - \frac {2} {\alpha } \ell.
\end{equation} 
Moreover, 
\begin{equation} \label{feLemCV2:2} 
\mu  = \lim  _{ r\to \infty   } r^{\frac {2} {\alpha }} f(r) 
\end{equation} 
exists and is finite.
In addition, if  $U\in \xCn{2} ( (0, \infty ) \times ( \R^N \setminus \{ 0\} ))$ is defined by
\begin{equation} \label{b1:fPS4:b3} 
U  (t, x)= t^{- \frac {1} {\alpha }} f  \Bigl( \frac { |x|} {\sqrt t} \Bigr), \quad t>0, x\not = 0,
\end{equation}
then the following properties hold.
\begin{enumerate}[{\rm (i)}] 

\item \label{eLemCV2:1} 
$U\in \xLn{{\alpha + 1}} _\Loc ((0, S ) \times \R^N ) \cap \mathrm{C} (( 0,S) \times ( \R^N \setminus \{0\}  ) ) $ is a solution of~\eqref{NLHE} in ${\mathcal D}' ((0,T) \times \R^N )$.

\item \label{eLemCV2:2} 
If $U_0 (x)= \mu  |x |^{-\frac {2} {\alpha }}$, then
\begin{equation} 
U(t) \goto  _{ t\to 0 } U_0,
\end{equation} 
in $\xLn{1} _\Loc (\R^N ) $.

\item \label{eLemCV2:3} 
If~\eqref{fCondAlpha} holds and $\ell = \beta ^{\frac {1} {\alpha }}$ in~\eqref{feLemCV2:1}, then  $U$ satisfies~\eqref{fPShrho} for all $t>0$ and $x\not = 0$. 
\end{enumerate} 
Moreover, there exist a sequence $(\mu_n )  _{ n\ge 1 } \subset (0,\infty )$, $\mu _n\to \infty $ and a sequence $(f_n) _{ n\ge 1 }$ of sign-changing solutions of the equation~\eqref{fpr2} for which  
$\displaystyle \lim  _{ r\to 0  } r^{\frac {2} {\alpha }} f_n(r)  =\beta ^{\frac {1} {\alpha }}$
and
$\displaystyle \lim  _{ r\to \infty   } r^{\frac {2} {\alpha }} f_n (r) =\mu _n$.
\end{prpstn} 

\begin{proof} 
The existence of the limit~\eqref{feLemCV2:1} follows from Propositions~3.2 and~3.3, and formula~(1.9) in~\cite{SoupletW}, 
then the limit~\eqref{feLemCV2:1:b1} follows from Proposition~3.1~(i) in~\cite{SoupletW}.
The existence of the limit~\eqref{feLemCV2:2} follows from Proposition~2.4 and formula~(1.9) in~\cite{SoupletW}.
Next, we show that $V(x)= U(1, x)= f(  |x| )$ is a solution of
\begin{equation} \label{fPR2} 
\Delta V + \frac {1} {2} x\cdot \nabla V + \frac {1} {\alpha } V +  |V |^\alpha V  = 0,
\end{equation} 
in ${\mathcal D}' (\R^N )$. Let  $\varphi \in \xCn{\infty } _\Comp (  \R^N)$ and $\varepsilon >0$. 
Since $  V \in  \xLn{{\alpha +1}} _\Loc (\R^N )$ by~\eqref{feLemCV2:1} and~\eqref{fCondAlpha}, we see that
\begin{equation} \label{fTPA62} 
 \Bigl| \int_{ \{  |x|<\varepsilon  \}} V    \Bigl( \Delta \varphi - \frac {1} {2} \nabla \cdot (x \varphi ) + \frac {1} {\alpha } \varphi  + | V |^\alpha  \varphi \Bigr) \Bigr|  \le C   \int_{ \{  |x|<\varepsilon  \}} ( |V | +  | V |^{ \alpha +1})\goto _{ \varepsilon \downarrow 0 } 0 .
\end{equation}  
Next, since $f\in \xCn{2} (0,\infty  ) $ is a solution of~\eqref{fpr2}, it follows that $V$ satisfies~\eqref{fPR2} in $\mathrm{C} ( \R^N \setminus \{ 0 \} )$, so that integration by parts yields
\begin{equation} \label{fTPA63} 
\int_{ \{  |x| >\varepsilon  \}} V \Bigl( \Delta \varphi - \frac {1} {2} \nabla \cdot (x \varphi ) + \frac {1} {\alpha } \varphi  + | V |^\alpha  \varphi \Bigr)       = -  \int_{ \{  |x|=\varepsilon  \}}  \Bigl( V  \frac {\partial \varphi } {\partial r}- \varphi \frac {\partial  V } {\partial r} -\frac {\varepsilon } {2}  V \varphi   \Bigr)    .
\end{equation} 
On the other hand, $r^{N-1} ( |   f (r) | +  |   f ' (r)|) \to 0$ as $r \downarrow 0$ by~\eqref{feLemCV2:1}, \eqref{feLemCV2:1:b1} and~\eqref{fCondAlpha}. Therefore, \begin{equation*} 
 \int_{ \{  |x|=\varepsilon  \}}  \Bigl(  |V |  + \Bigl|  \frac {\partial  V } {\partial r}  \Bigr| \Bigr)   \goto _{ \varepsilon \downarrow 0 } 0 ,
\end{equation*} 
and it follows from~\eqref{fTPA62} and~\eqref{fTPA63} that
\begin{equation*} 
\int_{\R^N}  V  \Bigl( \Delta \varphi - \frac {1} {2} \nabla \cdot (x \varphi ) + \frac {1} {\alpha } \varphi  + | V |^\alpha  \varphi \Bigr) \,  =0 . 
\end{equation*}  
Since $\varphi \in \xCn{\infty } _\Comp (  \R^N)$ is arbitrary, we see that $V$ solves~\eqref{fPR2}  in ${\mathcal D}' (\R^N )$.
Properties~\eqref{eLemCV2:1} and~\eqref{eLemCV2:2} now follows from~\cite[Lemma~7.1]{CDNW1}. 
Next, suppose~\eqref{fCondAlpha} and $\ell = \beta ^{\frac {1} {\alpha }}$ in~\eqref{feLemCV2:1}. It follows from Proposition~\ref{eLemCV1} with $a=1$ that $| r^{\frac {2} {\alpha }} f (r) - \beta ^{\frac {1} {\alpha }} |  \le C r^\rho$ for $r\le 1$. 
Moreover, it follows from~\eqref{feLemCV2:2}  that $| r^{\frac {2} {\alpha }} f (r) - \beta ^{\frac {1} {\alpha }} |  \le C$ for $r \ge 1$; and so,
\begin{equation*} 
| r^{\frac {2} {\alpha }} f (r) - \beta ^{\frac {1} {\alpha }} |  \le C  \Bigl( \frac {r} {r +1} \Bigr)^\rho , \quad r>0.
\end{equation*} 
Using~\eqref{b1:fPS4:b3}, this implies that 
\begin{equation*} 
| \,  |x|^\frac {2} {\alpha }U(t, x )-\beta ^{\frac {1} {\alpha }}| = 
\Bigl| \,   \Bigl(  \frac { |x|} {\sqrt t}  \Bigr) ^{\frac {2} {\alpha }} f \Bigl(  \frac { |x|} {\sqrt t}  \Bigr) - \beta ^{\frac {1} {\alpha }} \Bigr|  \le   C \Bigl(  \frac {  |x| } { |x|+ \sqrt t } \Bigr)^\rho  ,
\end{equation*}  
for $t>0$, $x\not = 0$.  
This proves Property~\eqref{eLemCV2:3}.
Finally, the last statement is proven in~\cite[Theorem~1.3]{CDNW1}.
\end{proof} 

\begin{proof} [Proof of Theorem~$\ref{ePSSZpr2}$]
Let $f\in \xCn{2} (0,\infty  ) $ be a solution of the equation~\eqref{fpr2}  having the singularity 
$r^{\frac {2} {\alpha }} f (r) \to  \beta ^{\frac {1} {\alpha }}$ as $r \to 0$. 
It follows from Proposition~\ref{eLemCV2} that the limit~\eqref{fEximu} exists and is finite, and that $U$ defined by~\eqref{fPS4:b3} satisfies~\eqref{ePSSZ2}, \eqref{fPShrho} (for all $t>0$, $x\not = 0$) and~\eqref{fDfnUdz} with $U_0 (x) = \mu  |x|^{- \frac {2} {\alpha }}$, and that $U$ is a solution of~\eqref{NLHE} in ${\mathcal D}' ((0,T) \times \R^N )$. 
The result now follows by applying Theorem~\ref{ePSSZ1} with $S=1$.
\end{proof} 

\appendix

\section{The heat equation with inverse square potential}
\label{sISP} 
In this section, we collect some properties of the homogeneous and nonhomogeneous heat equation with inverse square potential.
We assume~\eqref{fPShar}, and we consider the operator $H$ defined by~\eqref{fPShar:b1} and the corresponding semigroup $(e^{tH})  _{ t\ge 0 }$ on $\xLn{2} (\R^N ) $. 
We first recall some well-known properties of $H$.

\begin{lmm} \label{eLemSA1} 
Assume~\eqref{fPShar}, and let the operator $H$ on $\xLn{2 } (\R^N )$ be defined by~\eqref{fPShar:b1}.
\begin{enumerate}[{\rm (i)}] 

\item \label{eLemSA1:b1} 
$H $ is a self-adjoint, negative operator, and $D(H) \hookrightarrow  \xHn{1} (\R^N ) $ where $D(H) $ is equipped with the graph norm. 

\item \label{eLemSA1:b2} 
$H$ is the generator of a $C_0$ semigroup of contractions $(e^{t H}) _{ t\ge 0 }$, which is an analytic semigroup on $ \xLn{2 } (\R^N ) $.  

\item \label{eLemSA1:b3} 
Given $\DI \in \xLn{2 } (\R^N ) $, $u(t)= e^{tH} \DI$ satisfies $u\in \mathrm{C} ((0,\infty ), \xHn{1} (\R^N ) )$ and
\begin{equation} \label{feLemSA1:b3} 
 \| u(t) \| _{ \xHn{1} }\le C (1 + t^{- \frac {1} {2}} )  \| \DI \| _{ \xLn{2 } },
\end{equation} 
for all $t>0$, where the constant $C$ is independent of $\DI $.
\end{enumerate} 
\end{lmm} 

\begin{proof} 
Recall that 
\begin{equation} \label{fLSA1} 
 \Bigl\| \frac {u} { |\cdot |} \Bigr\|  _{ \xLn{2 } } \le \frac {2} {N-2}  \| \nabla u\| _{ \xLn{2 } },
\end{equation} 
and 
\begin{equation} \label{fLSA2} 
 \Bigl\| \frac {u} { |\cdot |^2} \Bigr\|  _{ \xLn{{\frac {2N} {N+2}}} } \le C \| \nabla u\| _{ \xLn{2 } },
\end{equation} 
for all $u\in \xHn{1} (\R^N ) $. Inequality~\eqref{fLSA1} is the standard Hardy estimate, and inequality~\eqref{fLSA2} is a Caffarelli-Kohn-Nirenberg inequality, see~\cite[inequality~(1.4)]{CaffarelliKN}.  
Since $\xLn{{\frac {2N} {N+2}}} (\R^N ) \hookrightarrow \xHn{{-1}} (\R^N ) $, it follows from~\eqref{fLSA2} that 
\begin{equation} \label{fLSA2:b1} 
L \in {\mathcal L} ( \xHn{1} (\R^N ) , \xHn{{-1}} (\R^N ) ),  \text{ where } 
Lu= \Delta u + \beta (\alpha +1)  |x|^{-2} u  .
\end{equation} 
Note that
\begin{equation*} 
Lu = Hu,\quad u\in D(H) .
\end{equation*} 
Since
\begin{equation} \label{fLSA3} 
\langle Lu , v\rangle  _{ \xHn{{-1}}, \xHn{1} } = - \int  _{ \R^N  }  \nabla u \cdot \nabla v +  \beta (\alpha +1) \int  _{ \R^N  } \frac { u v } {  |x|^2 }  ,
\end{equation} 
for all $u, v\in \xHn{1} (\R^N ) $,  it follows from~\eqref{fPShar} and~\eqref{fLSA1} that 
\begin{equation} \label{fLSA4} 
\langle Lu , u\rangle  _{ \xHn{{-1}}, \xHn{1} } = - \int  _{ \R^N  }  |\nabla u|^2 +  \beta (\alpha +1) \int  _{ \R^N  } \frac { |u|^2 } {  |x|^2 } \le - \nu  \int  _{ \R^N  }  |\nabla u|^2 ,
\end{equation} 
where $\nu = 1- \frac {4\beta (\alpha +1)}{(N-2)^2} >0$.
We deduce by Lax-Milgram's theorem that the map $u \mapsto - Lu +u$ is a homeomorphism $\xHn{1} (\R^N ) \to \xHn{{-1}} (\R^N )$. 
Let now $f\in \xLn{2} (\R^N ) $. It follows that there exists a unique $u\in \xHn{1} (\R^N ) $ such that $-Lu + u= f$. In particular, $Lu= u-f \in \xLn{2} (\R^N ) $, so that $u\in D(H) $ and $- Hu +u= f$. Thus we see that $R( I - H)= \xLn{2} (\R^N ) $. 
In addition, given $u,v\in D(H)$, we deduce from~\eqref{fLSA3} that
\begin{equation*} 
(H u, v)  _{ \xLn{2} }= \langle Lu , v\rangle  _{ \xHn{{-1}}, \xHn{1} } = \langle Lv , u\rangle  _{ \xHn{{-1}}, \xHn{1} } = (H v, u)  _{ \xLn{2} }.
\end{equation*} 
Thus we see that $H$ is symmetric. Moreover, it follows from~\eqref{fLSA4} that
\begin{equation*} 
(H u, u)  _{ \xLn{2} } \le 0,
\end{equation*} 
for all $u\in D(H)$. Therefore $H\le 0$ and $R( I-H) = \xLn{2} (\R^N ) $, so that $H$ is $m$-dissipative. 
Since $H$ is symmetric, we conclude that $H$ is a negative self-adjoint operator (see e.g.~\cite[Proposition~7.6 p.~193]{Brezis}).
In addition, it follows from~\eqref{fLSA4} that 
\begin{equation*} 
 \| u \| _{  \xHn{1 }  }^2 =  \|u \| _{ \xLn{2 } }^2 +  \|\nabla u\| _{ \xLn{2 } }^2 \le  \Bigl( 1+ \frac {1} {\nu} \Bigr) ( \|u \| _{ \xLn{2 } }^2 - (Hu, u) _{ \xLn{2 } }) \le 
 \Bigl( 1+ \frac {1} {\nu} \Bigr)  \|u\| _{ \xLn{2 } }  \|u \| _{ D(H) },
\end{equation*}  
 so that $D(H) \hookrightarrow  \xHn{1}(\R^N ) $.
This proves~\eqref{eLemSA1:b1}, and Property~\eqref{eLemSA1:b2} follows from the standard semigroup theory, see e.g.~\cite{Kato}, Chapter~IX, Theorem~1.24 and Example~1.25. 

Next, let $\DI \in \xCn{\infty } _\Comp (\R^N )$, so that $\DI \in D(H)$, and let $u(t) = e^{tH } \DI $, so that $u\in \mathrm{C} ([0, \infty ), D(H)) \cap \xCn{1} ([0, \infty ),  \xLn{2 } (\R^N ) ) \cap \xCn{\infty } ((0,\infty ), D(H))$ and $\partial _t u= H u$.  Multiplying the equation by $- Hu$, we obtain 
\begin{equation*} 
\frac {d} {dt} (- Hu, u) _{ \xLn{2 }  } = -  \|Hu \| _{  \xLn{2 }  }^2 \le 0,
\end{equation*} 
so that
\begin{equation} \label{fLSA5} 
(-H u(t), u(t) ) _{  \xLn{2 } } \le (-H u(s), u(s) ) _{ \xLn{2 } },\quad 0\le s\le t.
\end{equation} 
Next, multiplying the equation by $u$ and integrating on $(0,t) \times \R^N $,
\begin{equation*} 
\frac {1} {2}  \| u(t) \| _{ \xLn{2 } }^2 + \int _0^t (-Hu(s), u(s) ) _{ \xLn{2 } }  \xdif s = \frac {1} {2}  \| \DI \| _{ \xLn{2 } }^2,
\end{equation*} 
so that by~\eqref{fLSA5} 
\begin{equation*} 
2t (-H u(t), u(t) ) _{ \xLn{2 } } \le \| \DI \| _{ \xLn{2 } }^2 . 
\end{equation*} 
Using~\eqref{fLSA4} and the fact that $(e^{ t H }) _{ t\ge 0 }$ is a semigroup of contractions on $ \xLn{2 } (\R^N ) $, 
we conclude that~\eqref{feLemSA1:b3} holds. Since $\xCn{\infty } _\Comp (\R^N )$ is dense in $ \xLn{2 } (\R^N ) $, 
Property~\eqref{eLemSA1:b3} follows by a standard density argument. 
\end{proof} 

\begin{lmm} \label{eLemLHISP1} 
Given any $\DIbd \in  \xLn{2 } (\R^N ) $, $w(t)= : e^{tH} \DIbd $ satisfies $ w(t) \to \DIbd$ in $ \xLn{2 } (\R^N ) $ as $t\to 0$, and $w$ is a solution of
\begin{equation} \label{fLHISP} 
\partial _t w = \Delta w + \beta (\alpha +1)  |x|^{-2} w,
\end{equation} 
in $ \mathrm{C} ((0,T),  \xLn{2 } (\R^N ) )$, and in particular in ${\mathcal D}' ((0,\infty )\times \R^N)$. 
\end{lmm} 

\begin{proof} 
That $ w(t) \to \DIbd$ in $ \xLn{2 } (\R^N ) $ as $t\to 0$ follows from the fact that $(e^{tH}) _{ t\ge 0 }$ is a $C_0$ semigroup on $ \xLn{2 } (\R^N ) $. 
That the equation~\eqref{fLHISP} is satisfied in $\mathrm{C} ((0,T),  \xLn{2 } (\R^N ) )$ follows from the analyticity of the semigroup $(e^{tH} )  _{ t\ge 0 }$ on $\xLn{2 } (\R^N )$.
\end{proof} 

\begin{lmm} \label{eLemLHISP2} 
There exists $\varsigma >0$ such that 
\begin{equation*} 
 \Bigl\|  \frac {1} {h(t, x)}  \int  _{ \R^N  } \Kernel (t, x, y)   1 _{ \{  |y|>n \} } \DIbd  (y)  \xdif y  \Bigr\| _{ \xLn{\infty } (   |x| \le \frac {n} {2}) } \le   \Bigl(  1 + \frac {\sqrt t} {n} \Bigr)^\eta e^{-  \varsigma \frac {n^2} {t} }  \| \DIbd \| _{  \xLn{\infty } },
\end{equation*} 
for all $t>0$ and all $ \DIbd \in \xLn{\infty } (\R^N ) $.
\end{lmm} 

\begin{proof} 
It follows from~\eqref{fPS12} and~\eqref{fPSeta:2} that 
\begin{equation*} 
\begin{split} 
\frac {1} {h(t,x)} \int  _{ \R^N  } \Kernel (t, x, y)    1 _{ \{  |y|>n \} }  | \DIbd |  \, dy   & \le A  \| \DIbd \| _{ L^\infty  } t^{- \frac {N} {2}}  \int  _{   |y| >n }  e^{- \frac { |x-y| ^2 } { \cstpa  t}} h(t,y) \, dy \\ & \le \Bigl(   1 +  \frac {\sqrt t} {n} \Bigr)^\eta  A  \| \DIbd \| _{ L^\infty  } t^{- \frac {N} {2}}  \int  _{   |y| >n }  e^{- \frac { |x-y| ^2 } { \cstpa  t}}   dy   \\ & = \Bigl(  1 + \frac {\sqrt t} {n} \Bigr)^\eta  \Bigl(  \frac {4\pi} {a}  \Bigr)^{ \frac {N} {2}}  A  \| \DIbd \| _{ L^\infty  } e^{ \frac {\cstpa t} {4} \Delta }  1 _{ \{  |y|>n \}} .
\end{split} 
\end{equation*} 
Setting $z=\frac {x} {n}$ and $\tau =\frac {t} {n^2}$, we have  
\begin{equation*} 
( e^{ t \Delta }  1 _{ \{  |y|>n \} } ) (x) = ( e^{\tau \Delta }  1 _{ \{  |y|>1 \} } )  (z) =(4\pi  \tau )^{- \frac {N} {2}} \int  _{  |y|>1 } e^{ - \frac { |z-y|^2 } {4\tau }}.
\end{equation*} 
 If $|x|\le \frac {n} {2}$ then $|z|\le \frac {1} {2}$. For $ |y|\ge 1$ we have $  |z-y|\ge |y| -  |z| \ge  \frac { |y|} {2}$. Thus we see that for $ |x|\le \frac {n} {2}$
\begin{equation*} 
(4\pi )^{ \frac {N} {2}} ( e^{ t \Delta }  1 _{ \{  |y|>n \} } ) (x) =\tau ^{- \frac {N} {2}} \int  _{  |y|>1 } e^{ - \frac { |z-y|^2 } {4\tau }}\le  \tau ^{- \frac {N} {2}} \int  _{  |y|>1 } e^{ - \frac { | y|^2 } {16 \tau }} \le  e^{ - \frac {\varsigma } {\tau }}=e^{ - \frac {\varsigma n^2} {t}}
\end{equation*} 
for all $t>0$, where $\varsigma >0$. 
Hence the result follows.
\end{proof} 

\begin{lmm} \label{eLemLHISP4} 
For every $t>0$, the operator $ e^{t H} $ can be extended to a continuous operator $  \xLn{\infty}  (\R^N ) \to  \xLn{{\frac {2N} {N - 2}}} (\R^N ) + \xLn{\infty } (\R^N ) $.
Moreover, for  every $\DIbd \in \xLn{\infty } (\R^N )$, $ w(t) = e^{ t H} \DIbd $ satisfies the equation~\eqref{fLHISP}  in ${\mathcal D}' ((0,\infty )\times \R^N ) $, and
\begin{equation} \label{eLemLHISP4:1} 
 |w (t, x)| \le C \| \DIbd \| _{ \xLn{\infty }  }  h(t, x) .
\end{equation} 
In addition, $e^{ t H} \DIbd \to \DIbd$ in $\xLn{1} _\Loc  (\R^N ) $ as $t \to 0$. 
\end{lmm} 

\begin{proof} 
Let $\DIbd \in \xLn{ \infty } (\R^N )$. 
Set $\DIbd ^n = 1 _{ \{  |x|< n\} } \DIbd \in \xLn{2} (\R^N ) $  and  $w_n (t) = e^{t H } \DIbd ^n$. 
It follows from Lemma~\ref{eLemLHISP2} that for $m\ge n$
\begin{equation} \label{eLemLHISP4:2} 
\frac {1} {h}  | w_n (t) - w_m (t)  |  \le \frac {1} {h}
   \int  _{ \R^N  } \Kernel (t, x, y)   | 1 _{ \{ n < |x|< m\} }  \DIbd   | \xdif y   
 \le  \frac {1} {h} 
   \int  _{ \R^N  } \Kernel (t, x, y)    1 _{ \{  |x| > n\} }  | \DIbd   | \xdif y  \goto  _{ n\to \infty  } 0
\end{equation} 
in $\xLn {\infty } ((0,T) \times \{ |x|< T \} )$ for every $T >0$. 
Since $h  \in \xLn{ \infty } ((0,T),  \xLn {{ \frac {2N} {N-2} }}  ( \{ |x|< T \} ))$, we conclude easily that there exists $w\in \xLn {{\frac {2N} {N-2}}}_\Loc  ((0,\infty ) \times \R^N ) $ such that $w_n \to w$ as $n\to \infty $ in $\xLn {\infty }((0,T),  \xLn{{ \frac {2N} {N-2} }}  ( \{ |x|< T \} ))$ for every $T >0$.     
Moreover, since $w_n (t) \to 1 _{ \{  |x|< n\} }  \DIbd  $ in $\xLn{2} (\R^N ) $ as $t\to 0$, by Lemma~\ref{eLemLHISP1},
we deduce that $w(t) \to \DIbd $ in $\xLn{1} _ \Loc (\R^N ) $ as $t \to 0$.
Finally, it follows from~\eqref{fPS12} that 
\begin{equation*} 
 |    w_n (t)  |  \le (4\pi)^{ \frac {N} {2}} h(t)  A e^{ \frac {\cstpa t} {4} \Delta } (  h(t)  |1 _{ \{  |x| < n\} }  \DIbd | )  \le C  \| \DIbd \| _{ \xLn{\infty }  } h(t)   e^{ \frac {\cstpa t} {4} \Delta } (  h(t) ) .
 \end{equation*} 
Note that $h(t, y) \le C( 1 + t^{\frac {\eta } {2}}  |x|^{- \eta} )$. Moreover, since $e^{ \frac {\cstpa t} {4} \Delta }  (  |x|^{- \eta} )\le C (t+  |x|^2) ^{- \eta } \le C  t^{ - \frac {\eta } {2}}  $ by~\cite[Corollary~8.3]{CazenaveDEW}, we see that $  e^{ \frac {\cstpa t} {4} \Delta } (  h(t) ) \le C $ and~\eqref{eLemLHISP4:1} follows.

Let now $\theta \in \xCn{\infty } _\Comp ((0, \infty ) \times \R^N )$ and let $T>0$ be sufficiently large so that 
\begin{equation} \label{eLemLHISP4:3} 
\Supp \theta \subset (0,T) \times \{  |x|< T \}. 
\end{equation} 
Since $\DIbd ^n  \in \xLn{2} (\R^N ) $, it follows from Lemma~\ref{eLemLHISP1} that $w_n $ is a solution of~\eqref{fLHISP}  in ${\mathcal D}' ((0,\infty )\times \R^N )$; and so,
\begin{equation*} 
\begin{split} 
\int  _0^\infty  \int  _{ \R^N  } w (- \partial _t\theta - \Delta \theta - |x|^{-2} \theta  )  =&  \int  _0^\infty  \int  _{ \R^N  } w_n (- \partial _t\theta - \Delta \theta - |x|^{-2} \theta  )  \\ &+ \int  _0^\infty  \int  _{ \R^N  } (w-w_n ) (- \partial _t\theta - \Delta \theta - |x|^{-2} \theta  ) \\ = & \int  _0^\infty  \int  _{ \R^N  } (w-w_n ) (- \partial _t\theta - \Delta \theta - |x|^{-2} \theta  ) .
\end{split} 
\end{equation*}  
On the other hand, $h\in \xLn{1} ((0,T) \times \{  |x|<T \}$, so it follows from~\eqref{eLemLHISP4:2} and~\eqref{eLemLHISP4:3} that
\begin{equation*} 
 \int  _0^T \int  _{ \R^N  } (w-w_n ) (- \partial _t\theta - \Delta \theta - |x|^{-2} \theta  ) \goto _{ n\to \infty  } 0.
\end{equation*} 
Therefore, 
\begin{equation*} 
 \int  _0^\infty  \int  _{ \R^N  } w (- \partial _t\theta - \Delta \theta - |x|^{-2} \theta  ) =0,
\end{equation*} 
showing that  $w$ satisfies~\eqref{fLHISP}  in ${\mathcal D}' ((0,\infty )\times \R^N )$.
\end{proof} 

\begin{lmm} \label{eLemLHISP5} 
Let $T>0$, $f\in \xLn{\infty }   ((0,T), \xLn{2}  (\R^N ) )$, and set
\begin{equation} \label{eLemLHISP5:1} 
w(t) = \int _0 ^t e^{ (t-s) H} f(s) \xdif s,
\end{equation} 
for $0\le t\le T$. It follows that $w \in  \mathrm{C} ( [0,T], \xHn{1} (\R^N ) ) \cap \xWn{{1, \infty }} ( (0,T), \xHn{{-1}} (\R^N ) )$ and that
\begin{equation} \label{eLemLHISP5:2} 
\partial _t w - \Delta w - \beta (\alpha +1)  |x|^{-2} w= f ,
\end{equation} 
in $\xLn{\infty } (( 0,T), \xHn{{-1}} (\R^N ) )$, and in particular in ${\mathcal D}' ( (0,T) \times \R^N )$. 
\end{lmm} 

\begin{proof} 
Suppose first $f\in \xCn{\infty } _\Comp ((0,T) \times (\R^N  \setminus \{0\}) ) $. It follows in particular that $f \in  \mathrm{C} ([0,T], D(H))$ so that $w\in \mathrm{C} ( [0,T], D(H)) \cap \mathrm{C} ^1 ( [0,T], \xLn{2} (\R^N ) )$ and equation~\eqref{eLemLHISP5:2} holds in $\mathrm{C} ([0,T], \xLn{2}  (\R^N ) )$, see e.g.~\cite[Chapter~4, Corollary~2.6]{Pazy}.
In particular, $w\in \mathrm{C} [0,T], \xHn{1} (\R^N ) )$ and, using~\eqref{feLemSA1:b3},
\begin{equation}  \label{eLemLHISP5:3} 
 \| w \| _{ \xLn{\infty }  ((0,T), \xHn{1} (\R^N ) ) } \le C  \int _0 ^t  (1 + (t-s)^{- \frac {1} {2}} )  \| f(s) \| _{ \xLn{2 } } \xdif s
  \le C  \| f \| _{ \xLn{4}  ((0,T), \xLn{2} (\R^N ) ) }. 
\end{equation} 
Let now $f\in \xLn{\infty }  ((0,T), \xLn{2}  (\R^N ) )$ and $( f_n) _{ n\ge 1 } \subset \xCn{\infty} _\Comp ((0,T) \times (\R^N  \setminus \{0\}) )$ such that $f_n \to f$ in $\xLn{4}  ((0,T), \xLn{2} (\R^N )  )$.
Let $w_n$ be given by~\eqref{eLemLHISP5:1} with $f$ replaced by $f_n$. 
It follows that $w_n \to w$ in $ \mathrm{C} ([0,T], \xLn{2} (\R^N ) )$. 
Moreover, we deduce from~\eqref{eLemLHISP5:3} with $f$ replaced by $f_n- f_m$ that $w_n$ is a Cauchy sequence in $ \mathrm{C} ([0,T], \xHn{1} (\R^N ) )$. Therefore, $w\in  \mathrm{C} ([0,T], \xHn{1} (\R^N ) ) $ and $w_n \to w$ in $\mathrm{C} ([0,T], \xHn{1} (\R^N ) )$. 
Using~\eqref{fLSA2:b1}, we deduce that 
\begin{equation*} 
\Delta w_n + \beta (\alpha +1)  |x|^{-2} w_n \to \Delta w + \beta (\alpha +1)  |x|^{-2} w,
\end{equation*} 
in $\mathrm{C} ([0,T], \xHn{{-1}} (\R^N ) )$. Since $\partial _t w_n = Hw_n + f_n$, we conclude that
\begin{equation*} 
\partial _t w_n \goto _{ n\to \infty  } \Delta w + \beta (\alpha +1)  |x|^{-2} w +f
\end{equation*} 
in $\xLn{4}  ((0,T) , \xHn{{-1}} (\R^N ) )$. On the other hand, $\partial _t w_n \to \partial _t w$ in ${\mathcal D} ' ((0,T), H^1 (\R^N ) )$, so that~\eqref{eLemLHISP5:2} holds in $\xLn{4} ((0,T) , \xHn{{-1}} (\R^N ) )$. Since all terms in~\eqref{eLemLHISP5:2}, except perhaps $\partial _t w$, belong to $\xLn{ \infty } ((0,T) , \xHn{{-1}} (\R^N ) )$, we see that $\partial _t w \in \xLn{ \infty } ((0,T) , \xHn{{-1}} (\R^N ) )$ and that~\eqref{eLemLHISP5:2} holds in $\xLn{ \infty } ((0,T) , \xHn{{-1}} (\R^N ) )$.
\end{proof} 

\begin{lmm} \label{eLemLHISP7} 
Let $T>0$, $\alpha >0$ and $u,f\in \xLn{\infty } _\Loc ( (0,T) \times (\R^N \setminus \{0\} )$ satisfy
\begin{equation*} 
\partial _t u - \Delta u =  f 
\end{equation*} 
in ${\mathcal D}' ( (0,T) \times ( \R^N \setminus \{0\} ))$. 
It follows that $u \in  \mathrm{C} ((0,T) \times ( \R^N \setminus \{0\})$. 
\end{lmm} 

\begin{proof} 
This is standard interior parabolic regularity. For any $0 < \delta < \min\{ 1, \frac {T} {2} \}$, we define $ {\mathcal O}_\delta = (\delta , T- \delta ) \times \{ x\in \R^N ;\, \delta <  |x|< \frac {1} {\delta }\}$. 
Fix $p \in (N+1, \infty )$, so that $ \xWn{{1, p }} (  {\mathcal O}_\delta  ) \hookrightarrow \mathrm{C} (  \overline{ {\mathcal O}_\delta }  )$. Let $(\rho _n) _{ n\ge 1 }$ be a regularizing sequence with $\Supp \rho _n \subset \{ (t^2 +  |x|^2 )^{\frac {1} {2}}< \frac {1} {n} \}$. Fix $0 < \varepsilon  < \min\{ \frac {1} {2}, \frac {T} {4} \}$. It follows that for $n\ge \frac {1} {\varepsilon }$, $u_n=: \rho _n \star u$ and $f_n =: \rho _n \star f$, where the convolution is in $ \R^{1+ N} $, are well defined on $ {\mathcal O}_ \varepsilon $, and that $u_n \to u$ and $f_n \to f$ in $\xLn{p} (  {\mathcal O}_ \varepsilon  )$ as $n \to \infty $. Moreover, $u_n, f_n \in \xCn{\infty}  ( \overline{ {\mathcal O}_ \varepsilon }  )$ and $\partial _t u_n - \Delta u_n= f_n $ in $ {\mathcal O}_ \varepsilon $.
By parabolic interior regularity (see e.g.~\cite[Theorem~7.22]{Lieberman}), 
\begin{equation*} 
\begin{split} 
 \| u_n - u_m \| _{ \xLn{\infty }  ( {\mathcal O}_ {2\varepsilon } ) } & \le C \| u_n - u_m \| _{ \xWn{{1, p}} ( {\mathcal O}_ {2\varepsilon } ) } \\ & \le C ( \| f_n - f_m \| _{ \xLn{p} ( {\mathcal O}_ {\varepsilon } ) } + \| u_n - u_m \| _{ \xLn{p} ( {\mathcal O}_ {\varepsilon } ) } ) \goto _{ n,m \to \infty  } 0.
\end{split} 
\end{equation*} 
Therefore $u_n $ is a Cauchy sequence in $\xLn{\infty } ( {\mathcal O}_ {2\varepsilon } ) $, so that $u \in \mathrm{C} (  \overline{{\mathcal O}_ {2\varepsilon }}  ) $. The result follows by letting $\varepsilon \to 0$.
\end{proof}


\begin{thebibliography}{99}

\bibitem{BarasG}{P. Baras and J. Goldstein,} {The heat equation with a
singular potential, Trans. Amer. Math. Soc.  {\bf 284} (1984), no. 1, 121--139.}
\MScN{MR0742415} \DOI{10.1090/S0002-9947-1984-0742415-3}

\bibitem{Brezis}{H. Brezis,} {{\it Functional analysis, Sobolev spaces and partial differential equations}. Universitext. Springer, New York, 2011.}
\MScN{MR2759829} \DOI{10.1007/978-0-387-70914-7}

\bibitem{CaffarelliKN}{L. Caffarelli, R. V. Kohn and L. Nirenberg,} {First order interpolation inequalities with weights, Compositio Math. {\bf 53} (1984), no. 3, 259--275. }
\MScN{MR0768824} \LINK{http://www.numdam.org/item?id=CM\_1984\_\_53\_3\_259\_0}

\bibitem{CazenaveDEW}{T. Cazenave, F. Dickstein, M. Escobedo and F. B. Weissler,} {Self-similar solutions of a nonlinear heat
equation,  J. Math. Sci. Univ. Tokyo  {\bf 8} (2001), no.~3, 501--540.}
{\MScN{MR1855457}} {\LINK{http://journal.ms.u-tokyo.ac.jp/abstract/jms080305.html}}

\bibitem{CDNW1}{T. Cazenave, F. Dickstein, I. Naumkin and F. B. Weissler,} {Sign-changing self-similar solutions of the nonlinear heat equation with positive initial value. To appear in Amer. J. Math. Preprint: arXiv:1706.01403 [math.AP].} 
\LINK{http://arxiv.org/abs/1706.01403}

\bibitem{CDNW2}{T. Cazenave, F. Dickstein, I. Naumkin and F. B. Weissler,} {Perturbations of self-similar solutions. Dyn. Partial Differ. Equ.  {\bf 16} (2019), no. 2, 151--183.}
 \MScN{MR3928036} \DOI{10.4310/DPDE.2019.v16.n2.a3}
 
\bibitem{HoshinoY}{M. Hoshino and E. Yanagida,} {Convergence rate to singular steady states in a semilinear parabolic equation. Nonlinear Anal.  {\bf 131}  (2016), 98--111.}
\MScN{MR3427972} \DOI{10.1016/j.na.2015.06.020}

\bibitem{Kato}{T. Kato,} {{\it Perturbation theory for linear
operators}, Reprint of the 1980 edition. Classics in Mathematics. Springer-Verlag, Berlin, 1995.}
\MScN{MR1335452} \DOI{10.1007/978-3-642-66282-9}
 
\bibitem{Lieberman}{G. M. Lieberman,} {{\it Second order parabolic differential equations}. World Scientific Publishing Co., Inc., River Edge, NJ, 1996.}
\MScN{MR1465184} \DOI{10.1142/3302} 

\bibitem{LiskevichS}{V. A. Liskevich and Z. Sobol,} {Estimates of integral kernels for semigroups associated with second-order elliptic operators with singular coefficients. Potential Anal.  {\bf 18}  (2003), no. 4, 359--390.} 
\MScN{MR1953267} \DOI{10.1023/A:10218770259}

\bibitem{MazzeoP}{R. Mazzeo and F. Pacard,} {A construction of singular solutions for a semilinear elliptic equation using asymptotic analysis.  J. Differential Geom.  {\bf 44}  (1996), no. 2, 331--370.}
\MScN{MR1425579} \DOI{10.4310/jdg/1214458975}

\bibitem{MilmanS}{P. D. Milman  and Yu. A. Semenov,} {Heat kernel bounds and desingularizing weights.   J. Funct. Anal.  {\bf 202}  (2003), no. 1, 1--24.}
\MScN{MR1994762} \DOI{10.1016/S0022-1236(03)00018-1}

\bibitem{MoschiniT}{L. Moschini and A. Tesei,} {Parabolic Harnack inequality for the heat equation with inverse-square potential.  Forum Math.  {\bf 19}  (2007), no. 3, 407--427.}
\MScN{MR2328115} \DOI{10.1515/FORUM.2007.017}

\bibitem{Pazy}{A. Pazy,} {{\it Semi-groups of linear operators and
applications to partial differential equations}, 
Applied Math. Sciences {\bf 44}, Springer, New-York, 1983.}
\MScN{MR0710486} \DOI{10.1007/978-1-4612-5561-1}

\bibitem{QuittnerS}{P. Quittner and P. Souplet,}  {{\it Superlinear parabolic problems. 
Blow-up, global existence and steady states.} Second edition. Birkh\"auser Advanced Texts: Basler Lehrb\"ucher. [Birkh\"auser Advanced Texts: Basel Textbooks] Birkh\"auser/Springer, Cham, 2019.}
\MScN{MR3967048} \DOI{10.1007/978-3-030-18222-9}

\bibitem{Sato2011}{S. Sato,} {Blow-up at space infinity of a solution with a moving singularity for a semilinear parabolic equation.  Commun. Pure Appl. Anal.  {\bf 10}  (2011), no. 4, 1225--1237.}
\MScN{MR2787444} \DOI{10.3934/cpaa.2011.10.1225}

\bibitem{SatoY2009}{S. Sato and E. Yanagida,} {Solutions with moving singularities for a semilinear parabolic equation. J. Differential Equations  {\bf 246}  (2009), no. 2, 724--748.}
\MScN{MR2468735} \DOI{10.1016/j.jde.2008.09.004}

\bibitem{SatoY2010}{S. Sato and E. Yanagida,} {Forward self-similar solution with a moving singularity for a semilinear parabolic equation, Discrete Contin. Dynam. Systems  {\bf 26} (2010), n.1, 1313--331.}
\MScN{MR2552790} \DOI{10.3934/dcds.2009.26.313}

\bibitem{SatoY2012-1}{S. Sato and E. Yanagida,} {Appearance of anomalous singularities in a semilinear parabolic equation. Commun. Pure Appl. Anal.  {\bf 11}  (2012), no. 1, 387--405.}
\MScN{MR2833354} \DOI{10.3934/cpaa.2012.11.387}

\bibitem{SatoY2012-2}{S. Sato and E. Yanagida,} {Asymptotic behavior of singular solutions for a semilinear parabolic equation. Discrete Contin. Dyn. Syst.  {\bf 32}  (2012), no. 11, 4027--4043.}
\MScN{MR2945818} \DOI{10.3934/dcds.2012.32.4027}

\bibitem{SerrinZ}{J. Serrin and H. Zou,} {Classification of positive solutions of quasilinear elliptic equations.  Topol. Methods Nonlinear Anal.  {\bf 3}  (1994), no. 1, 1--25.}
\MScN{MR1272885} \DOI{10.12775/TMNA.1994.001}

\bibitem{SoupletW}{P. Souplet and F. B. Weissler F.B.:} {Regular self-similar solutions to
the nonlinear heat equation with initial data above the singular steady state,
Ann. Inst. H.~Poin\-ca\-r\'e Anal. Non Li\-n\'e\-ai\-re  {\bf 20} (2003),
213--235.}
\MScN{MR1961515} \DOI{10.1016/S0294-1449(02)00003-3}

\bibitem{VazquezZ}{V\'azquez J.L. and Zuazua E.:} {The Hardy inequality and the asymptotic behaviour of the heat equation with an inverse-square potential.  J. Funct. Anal. {\bf 173}  (2000), no. 1, 103--153.}
\MScN{MR1760280} \DOI{10.1006/jfan.1999.3556}

\end{thebibliography}
\end{document}